\documentclass[11pt]{amsart}

\usepackage{mathtools}
\usepackage{latexsym}
\usepackage{mathrsfs}
\usepackage{listings}
\usepackage{hyperref}
\usepackage{pgfplots}
\usepackage{appendix}
\hypersetup{
  unicode=true,
  pdfauthor={},
  pdftitle={Periodic/subharmonic solutions in second order difference equations:
            existence and bifurcation results},
  pdfsubject={},
  pdfkeywords={}; {}
}
\usepackage{pst-all}	
\usepackage{xcolor}
\usepackage{soul}
\usepackage{enumitem}
\usepackage{todonotes}
\usepackage{amssymb,bm}
\usepackage{amsmath}
\usepackage{amsfonts}
\usepackage{amsthm}
\usepackage{tikz}
\usepackage{epsfig}
\usepackage{verbatim}
\usepackage{float}
\usepackage{tabularx}
\usepackage{bm}
\definecolor{codegray}{rgb}{0.5,0.5,0.5}
\lstdefinestyle{mystyle}{
    commentstyle=\color{codegreen},
    keywordstyle=\color{magenta},
    numberstyle=\tiny\color{codegray},
    stringstyle=\color{codepurple},
    basicstyle=\ttfamily\footnotesize,
    breakatwhitespace=false,         
    breaklines=true,                 
    captionpos=b,                    
    keepspaces=true,                 
    numbers=left,                    
    numbersep=5pt,                  
    showspaces=false,                
    showstringspaces=false,
    showtabs=false,                  
    tabsize=2
}
  \lstdefinelanguage{GAP}{
    basicstyle=\ttfamily,
    keywords={true, false, function, return, fail, if, in, while, do, od, else, elif, fi, break, continue},
    keywordstyle=\color{blue}\bfseries,
    otherkeywords={
      >, <, ==
    },
    breaklines=true,      
    identifierstyle=\color{black},
    sensitive=True,
    comment=[l]{\#},
    commentstyle=\color{cyan},
    stringstyle=\color{red},
    morestring=[b]',
    morestring=[b]"
  }

\providecommand{\U}[1]{\protect\rule{.1in}{.1in}}

\newcolumntype{Y}{>{\raggedleft\arraybackslash}X}
\def\bc{{\mathbb{C}}}

\def\bn{{\mathbb{N}}}

\def\br{{\mathbb{R}}}

\def\bz{{\mathbb{Z}}}

\def\br{\mathbb R}

\def\wt{\widetilde}

\def\vs{\vskip.3cm}
\def\noi{\noindent}
\def\wt{\widetilde}
\def\gdeg{G\text{\rm -deg}}

\def\Om{\Omega}

\def\ve{\varepsilon}

\def\ker{\text{\rm Ker\,}}

\DeclareMathOperator{\id}{Id}

\newcommand\cV{\ensuremath{\mathcal V}}

\newrgbcolor{violet}{.6 .1 .8}
  \newrgbcolor{lightyellow}{1 1 .8}
  \newrgbcolor{lightblue}{.80 1 1}
  \newrgbcolor{mygreen}{0 .66 .05}
  \definecolor{mygreen}{rgb}{0,.66,.05}
  \definecolor{lightyellow}{rgb}{1,1,.80}
  \newrgbcolor{orange}{1 .6 0}
  \newrgbcolor{GreenYellow}{.85 1 .31}
  \newrgbcolor{Yellow}{1  1  0}
  \newrgbcolor{Goldenrod}{1  .90  .16}
  \newrgbcolor{Dandelion}{1  .71  .16}
  \newrgbcolor{Apricot}{1  .68  .48}
  \newrgbcolor{Peach}{1  .50  .30}
  \newrgbcolor{Melon}{1  .54  .50}
  \newrgbcolor{YellowOrange}{1  .58  0}
  \newrgbcolor{Orange}{1  .39  .13}
  \newrgbcolor{BurntOrange}{1  .49  0}
  \newrgbcolor{Bittersweet}{1.  .4300  .24}
  \newrgbcolor{RedOrange}{1  .23  .13}
  \newrgbcolor{Mahogany}{1.  .4475  .4345}
  \newrgbcolor{Maroon}{1.  .4084  .5376}
  \newrgbcolor{BrickRed}{1.  .3592  .3232}
  \newrgbcolor{Red}{1  0  0}
  \newrgbcolor{OrangeRed}{1  0  .50}
  \newrgbcolor{RubineRed}{1  0  .87}
  \newrgbcolor{WildStrawberry}{1  .04  .61}
  \newrgbcolor{CarnationPink}{1  .37  1}
  \newrgbcolor{Salmon}{1  .47  .62}
  \newrgbcolor{Magenta}{1  0  1}
  \newrgbcolor{VioletRed}{1  .19  1}
  \newrgbcolor{Rhodamine}{1  .18  1}
  \newrgbcolor{Mulberry}{.6668  .1180  1.}
  \newrgbcolor{RedViolet}{.9538  .4060  1.}
  \newrgbcolor{Fuchsia}{.5676  .1628  1.}
  \newrgbcolor{Lavender}{1  .52  1}
  \newrgbcolor{Thistle}{.88  .41  1}
  \newrgbcolor{Orchid}{.68  .36  1}
  \newrgbcolor{DarkOrchid}{.60  .20  .80}
  \newrgbcolor{Purple}{.55  .14  1}
  \newrgbcolor{Plum}{.50  0  1}
  \newrgbcolor{Violet}{.98 .15 .95}
  \newrgbcolor{RoyalPurple}{.25  .10  1}
  \newrgbcolor{BlueViolet}{.84  .38  .98}
  \newrgbcolor{Periwinkle}{.43  .45  1}
  \newrgbcolor{CadetBlue}{.38  .43  .77}
  \newrgbcolor{CornflowerBlue}{.35  .87  1}
  \newrgbcolor{MidnightBlue}{.4414  .9259  1.}
  \newrgbcolor{NavyBlue}{.06  .46  1}
  \newrgbcolor{RoyalBlue}{0  .50  1}
  \newrgbcolor{Blue}{0  0  1}
  \newrgbcolor{Cerulean}{.06  .89  1}
  \newrgbcolor{Cyan}{0  1  1}
  \newrgbcolor{ProcessBlue}{.04  1  1}
  \newrgbcolor{SkyBlue}{.38  1  .88}
  \newrgbcolor{Turquoise}{.15  1  .80}
  \newrgbcolor{TealBlue}{.1572  1.  .6668}
  \newrgbcolor{Aquamarine}{.18  1  .70}
  \newrgbcolor{BlueGreen}{.15  1  .67}
  \newrgbcolor{Emerald}{0  1  .50}
  \newrgbcolor{JungleGreen}{.01  1  .48}
  \newrgbcolor{SeaGreen}{.31  1  .50}
  \newrgbcolor{Green}{0  1  0}
  \newrgbcolor{ForestGreen}{.1992  1.  .2256}
  \newrgbcolor{PineGreen}{.3100  1.  .5575}
  \newrgbcolor{LimeGreen}{.50  1  0}
  \newrgbcolor{YellowGreen}{.56  1  .26}
  \newrgbcolor{SpringGreen}{.74  1  .24}
  \newrgbcolor{OliveGreen}{.6160  1.  .4300}
  \newrgbcolor{RawSienna}{.53  .28  .16}
  \newrgbcolor{Sepia}{1.  .7510  .70}
  \newrgbcolor{Brown}{.41  .25  .18}
  \newrgbcolor{TAN}{.86  .58  .44}
  \newrgbcolor{Gray}{1.  1.  1.}
  \newrgbcolor{Black}{1  1  1}
  \newrgbcolor{White}{1  1  1}

\newtheorem{theorem}{Theorem}[section]
\newtheorem{proposition}{Proposition}[section]
\newtheorem{lemma}{Lemma}[section]
\newtheorem{corollary}{Corollary}[section]

\newtheorem{definition}{Definition}[section]

\newtheorem{remark}{Remark}[section]

\newtheorem{remark-definition}{Remark and Definition}[section]
\newtheorem{rem-not}{Remark and Notation}[section]

\begin{document}

\title[Global Bifurcation of Non-Radial Solutions]{Global Bifurcation of Non-Radial Solutions for Symmetric Sub-linear Elliptic Systems on the Planar Unit Disc} 

\author{ Ziad Ghanem}\address{Department of Mathematical Sciences, University of Texas at Dallas, Richardson, TX 75080, USA}
\email{Ziad.Ghanem@UTDallas.edu}

\author{Casey Crane}\address{Department of Mathematical Sciences, University of Texas at Dallas, Richardson, TX 75080, USA}
\email{Casey.Crane@UTDallas.edu}

\author{Jingzhou Liu}\address{Department of Mathematical Sciences, University of Texas at Dallas, Richardson, TX 75080, USA}
\email{Jingzhou.Liu@UTDallas.edu}

\date{}

\maketitle

\begin{abstract}
In this paper, we prove a global bifurcation result for the existence of non-radial branches of solutions to the paramterized family of $\Gamma$-symmetric equations $-\Delta u=f(\alpha,z,u)$, $u|_{\partial D}=0$ on the unit disc $D:=\{z\in \bc : |z|<1\}$ with $u(z)\in \br^k$, 
where $\br^k$ is an orthogonal $\Gamma$-representation, $f: \br \times \overline D\times \br^k\to \br^k$ is a sub-linear $\Gamma$-equivariant continuous function, differentiable with respect to $u$ at zero and satisfying the conditions $f(\alpha, e^{i\theta}z,u)=f(\alpha, z,u)$ for all $\theta\in \br$ and $f(z,-u)=-f(z,u)$. 
\end{abstract}

\noi \textbf{Mathematics Subject Classification:} Primary: 37G40, 35B06; Secondary: 47H11, 35J57,  35J91, 35B32, 47J15 

\medskip

\noi \textbf{Key Words and Phrases:} global bifurcation, non-linear Laplace equation, symmetric bifurcation, bifurcation invariant, non-radial solutions, equivariant
Brouwer degree.


\section{Introduction}
We draw our motivation from natural phenomena that are modeled by parameterized systems of autonomous Partial Differential Equations (PDEs), where the possible states of each phenomenon are represented by solutions of the associated system of equations at corresponding parameter values. These solutions can be categorized as \textit{trivial} or \textit{non-trivial}.
Trivial solutions are so-named because they persist for all parameter values and their existence is straightforward to determine. In contrast, non-trivial solutions may only exist for a particular range of parameter values and often exhibit exotic properties that enhance our understanding of the model. The {\it classical bifurcation problem} is concerned with the existence and global behaviour of branches of non-trivial solutions emerging from their trivial counterparts.
Phenomena which admit the symmetries of a certain group $G$ (represented by the $G$-equivariance of the associated equations) may be studied as {\it equivariant bifurcation problems}, with additional consideration placed on the symmetric properties of these branches (cf. \cite{BBKX,survey,BCM,CLM,LWZ,R1,R4,R5,RY1,RY3}). 
\vs
The application of topological methods to the study of differential equations, as with most tools in the arsenal of the nonlinear analyst, traces back to M. Poincare \cite{HPAS}. More recently, the {\it Local/Equivariant Brouwer Degrees} and their infinite dimensional generalizations -- the {\it Local/Equivariant Leray-Schauder Degrees} --  have proven prodigiously effective in obtaining existence results for solutions to a wide class of nonlinear differential equations. Somewhat remarkably, these degree theories are also instrumental in solving classical and equivariant bifurcation problems. The first use of the Leray-Schauder degree for the detection of (local) bifurcation in a parameterized system of nonlinear differential equations is attributed to the seminal work of M.A. Krasnosel'skii, in which sufficient conditions for the existence of a branch of non-trivial solutions are established (cf. \cite{Kra3}). Only a decade later, P. Rabinowitz (cf. \cite{Rab}) proposed his famous {\it Rabinowitz alternative}, which provides sufficient conditions for the unboundedness of a branch of non-trivial solutions. In this paper, we employ equivariant analogues of Krasnosel'skii and Rabinowitz type results to solve an equivariant bifurcation problem.
\vs
Specifically, our objective is to study the equivariant bifurcation problem for a symmetric system of parameterized nonlinear Laplace equations subject to Dirichlet boundary conditions. The bifurcation and global behaviour of solutions for such systems has been extensively studied, due to their significance in applied mathematics and physics. In the case that the nonlinearity arises as the gradient of some potential function, the bifurcation problem can be studied as a parameterized variational problem and appropriate topological methods such as Morse theory, index theory, and the equivariant gradient degree (cf. \cite{GR}) can be administered (cf. \cite{Bar,BDZ,BdF,Chang,CFM,CLM,R1,R4,R5,RY1,RY3,W-Q}). 
However, these approaches are not applicable for systems which lack variational structure. The degree theoretic methods used in this paper impose no variational requirements on the equations and, as such, are broadly applicable to a wider class of problems.
\vs
Consider the following parameterized family of symmetric Laplace equations:
\begin{equation}\label{eq:Lap}
\begin{cases}
-\Delta u=f(\alpha, z,u), \quad \alpha \in \br, \; z \in D, \; u(z)\in V \\
u|_{\partial D}=0,
\end{cases}
\end{equation}
where $D:=\{z\in \bc: |z|<1\}$ is the unit planar disc, $V:=\br^k$ is an orthogonal representation of a finite group $\Gamma$ and $f:\br\times \overline D\times V\to V$ 
is a continuous, odd, radially symmetric and $\Gamma$-equivariant family of functions of sub-linear growth, which are differentiable with respect to $u$ at the origin in $V$. In particular, we assume that $f$ satisfies the following conditions:
\vs
  \begin{enumerate}[label=($A_\arabic*$)]
\item\label{c1}  $f(\alpha, e^{i\theta}z,u)=f(\alpha,z,u)$ \; for all $\alpha\in \br, \; z\in D$, \; $u\in V$ and $\theta \in \br$;
  \item\label{c2} $f(\alpha,z,\gamma u)=\gamma f(\alpha,z,u)$ \; for all $\alpha\in \br, \; z\in D$, \; $u\in V$ and $\gamma\in \Gamma$;
\item\label{c3}  $f(\alpha, z,-u)=-f(\alpha,z,u)$ \; for all $\alpha\in \br$, \; $z\in D$ and $u\in V$;           
\item\label{c4}  there exist continuous functions $A:\br \to L(k, \br)$, $c: \br \rightarrow (0,\infty)$ and a number $\beta>1$ such that for each $\alpha \in \br$ one has
\begin{align*}
  |f(\alpha, z,u)-A(\alpha) u|\le c(\alpha) |u|^\beta \quad \text{for all} \;
 {z\in \overline{D}}, \; {u\in V}; 
\end{align*}
\item\label{c5} there exist continuous functions $a,b: \br \rightarrow (0,\infty)$ and a number $\nu\in (0,1)$ such that for each $\alpha \in \br$ the following inequality holds
\begin{align*}
|f(\alpha,z,u)|<a(\alpha)|u|^\nu +b(\alpha) \quad \text{for all} \; {z\in \overline{D}}, {u\in V}.
\end{align*}
\end{enumerate}
Conditions \ref{c1} \ref{c2} and \ref{c3} imply 
the $O(2) \times \Gamma\times \mathbb Z_2$-symmetry of system \eqref{eq:Lap} in the sense that $f:\br\times \overline D\times V\to V$ is $O(2)$-invariant with respect to the $O(2)$-action $O(2) \times D \rightarrow D$ given by $(\theta,z) \rightarrow e^{i \theta}z$, $(\kappa, z) \rightarrow \overline{z}$ and $\Gamma \times \mathbb Z_2$-equivariant with respect to the $\Gamma \times \mathbb Z_2$ action $(\Gamma \times \mathbb Z_2) \times V \rightarrow V$ given by $(\gamma, \pm 1, u) \rightarrow \pm \gamma u$. On the other hand, conditions \ref{c4} and \ref{c5} guarantee differentiability at the origin and sublinearity, respectively.
\vs
\begin{remark}\rm  \label{rm0}
Since the zero function is a solution to 
\eqref{eq:Lap} for all values $\alpha \in \br$, all trivial solutions to the system are of the form $(\alpha,0)$. Non-trivial solutions to \eqref{eq:Lap}, on the other hand, fall into two categories, namely:
\begin{itemize}
    \item[(i)] {\it radial solutions}, which depend only on the magnitude $|z|$ of any disc element $z\in D$ (although less obvious than the trivial solution, radial solutions can often be identified using classical methods, eg. by reducing \eqref{eq:Lap} to a second order ODE);
    \item[(ii)] and {\it non-radial solutions}, which exhibit dependency on the angular variable.
\end{itemize}
Our objective is to identify branches of non-radial solutions to \eqref{eq:Lap}, describe their possible symmetric properties, and characterize their global behavior.
\end{remark}
\vs
The methods used in this paper are
inspired by \cite{BHKR}, where an existence result was obtained for a similar sublinear elliptic system using equivariant degree theory. For a more thorough exposition of these topics, we direct readers to the recent monograph \cite{book-new} or, alternatively, the older text \cite{AED}. 
\vs
The equivariant degree is intimately connected with the classical Brouwer degree. Although its application is relatively simple, technical difficulties arise when dealing with algebraic computations related to unfamiliar group structures. Many of these issues can be resolved with usage of the G.A.P. system, and the G.A.P. package {\tt EquiDeg} (created by Haopin Wu), which is available online at  \url{https://github.com/psistwu/equideg} (cf. \cite{Pin}).
\vs
In the remainder of this paper, we employ tools from the equivariant degree theory to determine $(i)$ under what conditions branches of non-trivial solutions may bifurcate from a trivial solution, $(ii)$ under what conditions these branches consist only of non-radial solutions and $(iii)$ under what conditions these branches are unbounded. Subsequent sections are organized as follows: In Section \ref{sec:functional_reformulation} the problem \eqref{eq:Lap} is reformulated in an appropriate functional setting. In Section \ref{sec:bifurcation_abstract} we recall the abstract equivariant bifurcation results, including the equivariant analogues of the classical Krasnosel'skii and Rabinowitz theorems. In Section \ref{sec:bifurcation}, we apply the equivariant degree theory methods to establish local and global bifurcation results for \eqref{eq:Lap}. Finally, in Section \ref{sec:example} we present a motivating example of vibrating membranes where our main results, Theorem \ref{th:bounded} and \ref{th:unbounded-K}, are applied to demonstrate the existence of unbounded branches of non-radial solutions admitting all possible maximal orbit types. For convenience, the Appendices include an explanation of notations used, a summary of the spectral properties of the Laplace operator on the unit disc, and a brief introduction to the Brouwer equivariant degree theory.
\vs
\noi{\bf Acknowledgment:} We are deeply grateful to our advisor, Professor Krawcewicz, whose invaluable guidance and unwavering support made this work possible. We would also like to acknowledge Professor Garcia-Azpetia for insightful discussions regarding the estimation of the behaviour of our system near a bifurcation point. 

\section{Functional Space Reformulation and {\em a priori} Bounds}\label{sec:functional_reformulation}
Consider the Sobolev space $\mathscr H:=H^2(D,V) \cap H^1_0(D,V)$ equipped with the usual norm 
\begin{align*}
    \|u\|_{\mathscr H}:=\max\{\|D^s u\|_{2}: |s|\le 2\},
\end{align*}
where $s := (s_1,s_2)$, $|s|:=s_1+s_2 \leq 2$, and $D^s \varphi :=\frac{\partial ^{|s|} \varphi}{\partial^{s_1}x\partial^{s_2}y}$. It is well known that the \textit{Laplacian operator} $\mathscr L:\mathscr H\to L^2(D,V)$ given by
\begin{align}
\label{eq:operator-L}
\mathscr Lu:=-\Delta u,
\end{align}
is a linear isomorphism. 
\vs
Let $\nu \in (0,1)$ be the scalar from condition \ref{c5}. If one chooses $q > \max\{1,2\nu\}$ (for example, it is enough to take $q:= 2\beta$ (cf. assumption \ref{c4}),
then there is the standard Sobolev embedding $j:\mathscr H\to L^q(D,V)$
\begin{equation}\label{eq:operator-j}
 j(u)(z) := u(z), \quad u \in \mathscr H,
\end{equation}
and also its associated \textit{Nemytski operator} $N_{\alpha}: L^q(D,V) \rightarrow L^2(D,V)$ 
\begin{equation}\label{eq:operator-N}
N_\alpha(v)(z) := f(\alpha, z,v(z)), \quad z\in D, \; \alpha \in \br.
\end{equation}
\begin{lemma} 
\label{lemm:operator_N_WD}\rm 
Under conditions \ref{c1}--\ref{c5}, The Nemystiki operator \eqref{eq:operator-N} is well-defined.
\end{lemma}
\begin{proof}
It suffices to demonstrate that for any
$v \in L^q(D,V)$ and $\alpha \in \br$ one has $N_\alpha(v) \in L^2(D,V)$.
Combining  
\ref{c5} with the H\"older inequality, one has:
\begin{align}
\|N_\alpha(v)\|_{L^2}& = \|f(\alpha,z,v) \|_{L^2} \nonumber\\
&\le  \|\, a(\alpha) |v|^\nu\|_{L^2} + \|\, b(\alpha) \|_{L^2}  \nonumber\\
&= a(\alpha) \|\, |v|^\nu\|_{L^2} +  b(\alpha) \sqrt \pi  \nonumber\\
& =   a(\alpha) \left(\int_D|v|^{2\nu}    \right)^{\frac 12} +  b(\alpha) \sqrt \pi \nonumber  \\
& \le a(\alpha) \pi^{{1 - 2\nu/q} \over 2} 
\left(\int_D |v|^{\frac{2\nu q}{2\nu }}    \right)^{\frac{2\nu}{2q}} + b(\alpha) \sqrt \pi \nonumber\\
&= a(\alpha) \pi^{1/2-\nu/q}\|v\|^\nu_{L^q} + b(\alpha) \sqrt{\pi},
\label{eq:estimate}
\end{align}
where the result follows from $b(\alpha),a(\alpha),\|v\|_{q} < \infty$.
\end{proof}
\vs
Notice that the equation
\begin{equation}\label{eq:s2-1}
\mathscr L u=N_\alpha(j u),\quad \alpha \in \br, u\in \mathscr H,
\end{equation}
is equivalent to system \eqref{eq:Lap} in the sense that
$u \in \mathscr H$ is a solution to \eqref{eq:Lap} for some $\alpha \in \br$ if and only if $(\alpha,u) \in \br \times \mathscr H$ satisfies \eqref{eq:s2-1}. In turn, the invertibility of $\mathscr L$ together with Lemma \eqref{lemm:operator_N_WD} imply that the map 
$\mathscr F:\br\times \mathscr H\to \mathscr H$ given by
\begin{equation}\label{eq:Lap1}
\mathscr F(\alpha, u):=u-\mathscr L^{-1}N_\alpha(j u),
\end{equation}
is well-defined. Hence, system \eqref{eq:Lap} is also equivalent to the equation
\begin{equation}\label{eq:bif}
    \mathscr F(\alpha,u) = 0,\quad \alpha\in \br,\; u\in \mathscr H.
\end{equation}
In what follows, we will call \eqref{eq:bif} the \textit{operator equation} associated with \eqref{eq:Lap}.

\begin{lemma}\label{lem:a-priori} \rm
Let $f:\br \times \overline{D} \times V \rightarrow V$ be a continuous function satisfying the assumption \ref{c5}. For every $\alpha \in \br$, there exists a constant $R(\alpha)>0$ such that, if $(\alpha,u) \in \br \times \mathscr H$ is a solution to system \eqref{eq:Lap}, then $\|u\|_{\mathscr H} < R(\alpha)$.
\end{lemma}
\begin{proof}
Assume that $(\alpha,u) \in \br \times \mathscr H$ is a solution to system \eqref{eq:Lap}. Combining \eqref{eq:estimate} with $u = \mathscr L^{-1}N_\alpha(ju)$, one has    
\begin{equation}\label{eq:est1}
\|u\|_{\mathscr H} \le a(\alpha) \pi^{1/2-\nu/q}\|\mathscr L^{-1}\|\, \|u\|_{L_2}^\nu + b(\alpha) \sqrt{\pi}\|\mathscr L^{-1}\|,
\end{equation}
which, together with $\|u\|_{\mathscr H} \ge \|u\|_{L^2}$, implies
\begin{align}\label{eq:s2-3}
\|u\|_{L^2}\le c\|u\|^{\nu}_{L^2}+d,
\end{align}
where $c:= a(\alpha) \pi^{1/2-\nu/q}\|\mathscr L^{-1}\|$, $d:= b(\alpha) \sqrt{\pi}\|\mathscr L^{-1}\|$.
Finally, since $0<\nu<1$, there exists $R_0(\alpha)>0$ such that $\psi(t):=t-ct^\nu-d>0$ for $t\ge R_0(\alpha)$. Consequently, $\|u\|_{L^2}<R_0(\alpha)$, and by \eqref{eq:est1},
\[
\|u\|_{\mathscr H} \le  c\|u\|^{\nu}_{L^2}+d< cR_0(\alpha)^\nu+d=:R(\alpha)
\]
is the required constant.
\end{proof}
\vs
\section{Abstract Local and Global Equivariant Bifurcation} \label{sec:bifurcation_abstract}
In this section we present a concise exposition, following \cite{book-new} and \cite{AED}, of an equivariant Brouwer degree method to study symmetric bifurcation problems. Given a compact Lie group 
$\mathcal G$, an isometric Banach $\mathcal G$-representation $\mathcal H$ and a completely continuous 
$\mathcal G$-equivariant field $\mathcal F: \br \times \mathcal H \rightarrow \mathcal H$, 
we use the Leray-Schauder Equivariant $\mathcal G$-Degree to describe local and global properties of the solution set to the equation,
\begin{equation}\label{eq:bif_loc1}
\mathcal F(\alpha,u)=0,\quad \alpha\in \br,\; u\in \mathcal H.
\end{equation}
\vs
To simplify our exposition and for compatibility with the system of interest \eqref{eq:bif}, we make the following two assumptions.
\vs

\begin{enumerate}[label=($B_\arabic*$)]
    \item\label{b1} The set of \textit{trivial solutions} to \eqref{eq:bif_loc1} is given by
    \begin{align*}
  M:= \{ (\alpha,0) \in \br \times \mathcal H \}.
\end{align*}
\item\label{b2} There exists a continuous family of linear operators $\mathcal A(\alpha): \mathcal H \rightarrow \mathcal H$ such that, for every $\alpha \in \br$, the derivative $D_u \mathcal F(\alpha,0)$ exists and
\begin{align*}
\mathcal A(\alpha) := D \mathcal F(\alpha,0).
\end{align*}
\end{enumerate}
\vs
Solutions
to \eqref{eq:bif_loc1} which do not belong to $M$ are called \textit{nontrivial}. Let $\mathcal S \subset \mathcal F^{-1}(0)$ denote the set of all non-trivial solutions, i.e.
\[
\mathcal S := \{ (\alpha,u) \in \br \times \mathcal H \; : \; \mathcal F(\alpha,u)=0 \text{ and } u \neq 0 \}.
\]
Clearly, the set of non-trivial solutions $\mathcal S \subset \br \times \mathcal H$ is $\mathcal G$-invariant. 

\subsection{The Local Bifurcation Invariant and Krasnosel'skii's Theorem}\label{sec:local-bif-inv}
Formulation of a Krasnosel'skii type local bifurcation result for equation \eqref{eq:bif_loc1} necessitates the introduction of additional notations and terminology
(for more details, the reader is referred to \cite{BBKX,book-new}). Our first definition clarifies what is meant by a bifurcation of the equation \eqref{eq:bif_loc1}.
\begin{definition}\rm
 A trivial solution $(\alpha_0,0) \in M$ is said to be a \textit{bifurcation point} for the equation \eqref{eq:bif_loc1} if every open neighborhood of the point $(\alpha_0,0)$ has non-trivial intersection with $\mathcal S$.   
\end{definition}
It is well-known that a necessary condition for any trivial solution $(\alpha_0,0) \in M$ to be a bifurcation point for the equation \eqref{eq:bif_loc1} is that the linear operator $\mathcal A(\alpha_0):\mathcal H \rightarrow \mathcal H$ is not an isomorphism. This leads to the following definition.
\begin{definition}\rm
A trivial solution $(\alpha_0,0) \in M$ is said to be a \textit{regular point} for the equation \eqref{eq:bif_loc1} if $\mathcal A(\alpha_0)$ is an isomorphism and a \textit{critical point} otherwise. Moreover, a critical point $(\alpha_0,0) \in M$ is said to be \textit{isolated} if there exists a deleted $\epsilon$-neighborhood $0< \vert \alpha - \alpha_0 \vert < \epsilon$ such that for all $\alpha \in (\alpha_0 - \epsilon, \alpha_0 + \epsilon) \setminus \{\alpha_0\}$, the point $(\alpha,0) \in M$ is regular.
\end{definition}
The set of all critical points for equation \eqref{eq:bif_loc1}, denoted $\Lambda$, is called the 
{\it critical set}, i.e.
\begin{align}\label{eq:critical}
 \Lambda:=\{(\alpha,0): \text{ $\mathcal A(\alpha):\mathcal H \rightarrow \mathcal H$ is not an isomorphism}\}.   
\end{align}
The next definition concerns our interest in the continuation of non-trivial solution emerging from a bifurcation point $(\alpha_0,0) \in M$.
\begin{definition}\rm
 A trivial solution $(\alpha_0,0) \in M$ is said to be a {\it branching point} for the equation \eqref{eq:bif_loc1} if there exists a non-trivial continuum $K \subset \overline{S}$ with $K \cap M = \{ (\alpha_0,0) \}$ and the maximal connected set $\mathcal C \subset \overline{S}$ containing the branching point $(\alpha_0,0)$ we call a \textit{branch} of nontrivial solutions bifurcating from the point $(\alpha_0,0)$.    
\end{definition}
Whereas the classical Krasnosiel'skii bifurcation result is only concerned with the existence of a branch of nontrivial solutions for the equation \eqref{eq:bif_loc1} bifurcating from a given critical  point $(\alpha_0,0) \in \Lambda$, the equivariant Krasnosiel'skii bifurcation result, which we employ in this paper, is also concerned with the symmetric properties of such a branch.
\begin{definition} \rm
Given a subgroup $H \leq \mathcal G$, denote by $\mathcal S^H$ the corresponding $H$-fixed point space of non-trivial solutions.
A branch of solutions $\mathcal C$ is said to have \textit{symmetries at least} $(H)$ if $\mathcal C \cap \mathcal S^H \not=\emptyset$.  
\end{definition}
Let $(\alpha_0,0) \in \Lambda$ be an isolated critical point with a deleted $\epsilon$-neighborhood
\[
\{ \alpha \in \br : 0 < | \alpha - \alpha_0 | < \epsilon \},
\]
on which $\mathscr A(\alpha): \mathscr H \rightarrow \mathscr H$ is an isomorphism and choose $\alpha^\pm_0 \in (\alpha_0 - \ve, \alpha_0 + \ve)$ with $\alpha^-_0 \leq \alpha_0 \leq \alpha^+_0$. Since $\mathcal A(\alpha^\pm): \mathscr H \rightarrow \mathscr H$ are non-singular, there exists a number $\delta >0$ sufficiently small such that, adopting the notations $\mathcal F_{\pm}(u) := \mathcal F(\alpha^\pm_0, u)$, $B_{\delta} := \{u\in \mathscr H: \|u\|< \delta\}$, one has
\[
\mathcal F_{\pm}^{-1}(0) \cap \partial B_{\delta} = \emptyset,
\]
and $\mathcal F_{\pm}$ are $B_{\delta}$-admissibly $\mathcal G$-homotopic to $\mathcal A(\alpha^\pm)$. Moreover, since $\mathcal G$ acts isometrically on $\mathscr H$, the ball $B_{\delta}$ is clearly $\mathcal G$-invariant. It follows, from the homotopy property of the $\mathcal G$-equivariant Leray-Schauder degree (cf. Appendix \ref{subsec:G-degree}), that $(\mathcal F_{\pm},B_{\delta})$ are admissible $\mathcal G$-pairs in $\mathscr H$ and also that $\mathcal G\text{-deg}(\mathcal F_{\pm},B_{\delta}) = \mathcal G\text{-deg}(\mathcal A(\alpha^\pm_0), B(\mathscr H))$, where $B(\mathscr H)$ is the open unit ball in $\mathscr H$. We 
call the Burnside Ring element
\begin{align} \label{def:local_bifurcation_invariant}
 \omega_{\mathcal G}(\alpha_0)=\mathcal G\text{-deg}(\mathcal A(\alpha^-_0), B(\mathscr H))-\mathcal G\text{-deg}(\mathcal A(\alpha^+_0), B(\mathscr H)), 
\end{align}
the {\it local bifurcation invariant} at $(\lambda_0,0)$. The reader is referred to \cite{book-new} or \cite{AED} for proof that the invariant \eqref{def:local_bifurcation_invariant} does not depend on the choice of $\alpha^\pm_0 \in \br$ or radius $\delta >0$, and also for the proof of the following local bifurcation result, which is a consequence of the equivariant version of a classical result of K. Kuratowski (cf. \cite{Kura}, Thm. 3, p. 170).
 \vs
 \begin{theorem}\em
 {\bf(M.A. Krasnosel'skii-Type Local Bifurcation)}\label{th:Kras}
 Let $\mathcal F: \br \times \mathscr H \rightarrow \mathscr H$ be a completely continuous $\mathcal G$-equivariant field satisfying the assumptions \ref{b1} and \ref{b2} and with 
 an isolated critical point $(\alpha_0,0)$. If $\omega_{\mathcal G}(\alpha_0) \neq 0$, then
 \begin{itemize}
     \item[(i)] there exists a branch of nontrivial solutions $\mathcal C$ to system \eqref{eq:bif_loc1} with branching point $(\alpha_0,0)$;
     \item[(ii)] moreover, if $(H) \in \Phi_0(G)$ is an orbit type with
     \begin{align*}
     \operatorname{coeff}^{H}
         (\omega_{\mathcal G}(\alpha_0)) \neq 0,
     \end{align*}
 \end{itemize}
then there exists a branch of non-trivial solutions bifurcating from $(\alpha_0,0)$ with symmetries at least $(H)$.
 \end{theorem}
 \vs

\subsection{Global Bifurcation and the Rabinowitz Alternative}\label{sec:rab-alt}
In order to employ the Leray-Schauder $\mathcal G$-equivariant degree to describe the global properties of a branch of non-trivial solutions bifurcating from an isolated critical point of the equation \eqref{eq:bif_loc1}, we need to make an additional assumption. 
\begin{enumerate}[label=($B_\arabic*$)]
\setcounter{enumi}{2}
\item\label{b3} The critical set $\Lambda \subset M$ (given by \eqref{eq:critical}) is discrete.
\end{enumerate}
Notice that the local bifurcation invariant $\omega_{\mathcal G}(\alpha_0)$ at any critical point $(\alpha_0,0) \in \Lambda$ is well-defined under assumption \ref{b3}. Moreover, if $\mathcal U \subset \br \times \mathscr H$ is an open bounded $\mathcal G$-invariant set, then its intersection with the critical set is finite. These observations will be important for the statement of the following global bifurcation result, the proof of which can be found in \cite{book-new}, \cite{AED}.
\vs
\begin{theorem}\label{th:Rabinowitz-alt}{\bf (The Rabinowitz Alternative)} \rm
Suppose that $\mathcal F: \br \times \mathscr H \rightarrow \mathscr H$ is a completely continuous $\mathcal G$-equivariant field satisfying conditions  \ref{b1}--\ref{b3} and let $\mathcal U \subset \br \times \mathscr H$  be an open bounded $\mathcal G$-invariant set with $\partial \mathcal U \cap \Lambda = \emptyset$. If $\mathcal C$ is a branch of nontrivial solutions to \eqref{eq:bif_loc1} bifurcating from the critical point $(\alpha_0,0) \in \mathcal U \cap \Lambda$, then one has the following alternative:
\begin{enumerate}[label=$(\alph*)$]
\item  either $\mathcal C \cap \partial \mathcal U \neq \emptyset$;
    \item or there exists a finite set
    \begin{align*}
        \mathcal C \cap \Lambda = \{ (\alpha_0,0),(\alpha_1,0), \ldots, (\alpha_n,0) \},
    \end{align*}
    satisfying the following relation
    \begin{align*}
        \sum\limits_{k=1}^n \omega_{\mathcal G}(\alpha_k) = 0.
    \end{align*}
\end{enumerate} 
\end{theorem}
\vs
\begin{remark} \rm
Suppose that Theorem \ref{th:Kras} is used to demonstrate the existence of a branch $\mathcal C$ of nontrivial solutions to \eqref{eq:bif_loc1} bifurcating from a critical point $(\alpha_0,0)$ and that certain conditions are met such that, for any open bounded $\mathcal G$-invariant neighborhood $\mathcal U \ni (\alpha_0,0)$ with $\partial \mathcal U \cap \Lambda = \emptyset$, the alternative \ref{alt_b} is impossible. Then, according to Theorem \ref{th:Rabinowitz-alt}, the branch $\mathcal C$ must be unbounded.
\end{remark}
\vs
\section{Local and Global  Bifurcation of Non-radial Solutions in \eqref{eq:Lap}} \label{sec:bifurcation}
Returning to the functional reformulation of our original problem described in Section \ref{sec:functional_reformulation}, consider the product group $G := O(2) \times \Gamma \times \mathbb{Z}_2$ and notice that the Sobolev space $\mathscr H$ is a natural Banach $G$-representation with respect to the isometric $G$-action $G \times \mathscr H \rightarrow \mathscr H$ given by
\begin{align} \label{def:isometric_action}
(\theta,\gamma,\pm1)u(z)&:=\pm \gamma u(e^{i\theta} \cdot z), \quad z\in D, \;\; u\in \mathscr H; \\
(\kappa,\gamma,\pm 1) u(z)&:=\pm \gamma u(\overline{z}), \quad \gamma\in \Gamma, \; \kappa \in O(2), \; \theta \in SO(2), \nonumber
\end{align}
where $\overline{z}$ is the complex conjugation of $z \in D$ and $e^{i\theta} \cdot z$ is the standard complex multiplication. 
\vs
\begin{remark}\rm  \label{rm1}
Under assumptions \ref{c1}-\ref{c5}, the nonlinear operator $\mathscr F: \br \times \mathscr H \rightarrow \mathscr H$ given by \eqref{eq:Lap1} is a completely continuous $G$-equivariant field, differentiable at $0\in \mathscr H$ with 
\begin{align*}
    D\mathscr F(\alpha,0)=  \operatorname{Id} -\mathscr L^{-1} A(
\alpha):\mathscr H \rightarrow \mathscr H,
\end{align*}
where $A(\alpha): V \rightarrow V$ is the linearization of the map $f(\alpha,z,u)$ from the equation \eqref{eq:Lap} at the origin (see for example see \cite{Kra1,BHKR}). For convenience, we introduce the notation
\begin{equation}\label{eq:Linearization1}
\mathscr A(\alpha)(u):=u-\mathscr L^{-1} A(
\alpha)u, \quad u\in \mathscr H.
\end{equation}
\end{remark}
\vs
Before approaching the question of bifurcation in the equation \eqref{eq:Lap}, we must derive a workable formula for the computation of the degree $\gdeg(\mathscr A(\alpha), B(\mathscr H))$, where $B(\mathscr  H):= \{ u \in \mathscr H : \| u \|_{\mathscr H} < 1 \}$ is the open unit ball in $\mathscr H$ and $\alpha \in \br$ is any parameter value for which $\mathscr A(\alpha): \mathscr H \rightarrow \mathscr H$ is an isomorphism.
\vs
Assuming that a complete list of the irreducible $\Gamma$-representations $\{ \mathcal V_j \}_{j=1}^r$ is made available we denote by $\{ \mathcal V_j^- \}_{j=1}^r$
the corresponding list of irreducible $\Gamma \times \mathbb{Z}_2$-representations, where the superscript is meant to indicate that each irreducible $\Gamma$-representation has been equipped with the antipodal $\mathbb{Z}_2$-action in the standard way (cf. \cite{book-new}, \cite{AED}). As a $\Gamma$-representation, $V = \br^k$ is also a natural $\Gamma \times \mathbb{Z}_2$-representation with the $\Gamma \times \mathbb{Z}_2$-isotypic decomposition
\begin{align*}
  V = V_1 \oplus V_2 \oplus \cdots \oplus V_r,  
\end{align*}
where each $\Gamma \times \mathbb{Z}_2$-isotypic component component $V_j$ is modeled on the irreducible $\Gamma \times \mathbb{Z}_2$-representation $\mathcal V_j^-$ in the sense that $V_j$ is equivalent to the direct sum of some number of copies of $\mathcal V_j^-$, i.e.
\begin{align*}
    V_j \simeq \mathcal V_j^- \oplus \cdots \oplus \mathcal V_j^-.
\end{align*}
The exact number of irreducible $\Gamma$-representations $\mathcal V_j^-$ `contained' in the $\Gamma \times \mathbb{Z}_2$-isotypic component $V_j$ is called the \textit{$\mathcal V_j^-$-isotypic multiplicity} of $V$ and is calculated according to the ratio,
\begin{align*}
 m_j:= \dim V_j / \dim \mathcal{V}_j^-, \quad j \in \{1,2,\ldots,r\}. 
\end{align*}
To simplify our computations, we introduce an additional condition on the linearization of \eqref{eq:Lap}:
\begin{enumerate} [label=($A_0$)] 
    \item\label{c0} For each $j \in \{1,2, \ldots, r\}$ there exists a continuous map $\mu_j: \br \rightarrow \br$ with
    \[
    A_j(\alpha) := A(\alpha)|_{V_j} = \mu_j(\alpha) \operatorname{Id}|_{V_j}.
    \]
\end{enumerate}
On the other hand, for every $m \in \bn$ we denote by $\mathcal W_m \simeq \bc$ the irreducible $O(2)$-representation equipped with the $O(2)$-action
\[
(\theta,z) := e^{i m \theta} \cdot z, \; (\kappa,z) := \overline{z}, \; z \in \mathcal W_m,
\]
and by $\mathcal W_0 \simeq \br$ the irreducible $O(2)$-representation with the trivial $O(2)$-action.
\vs
In pursuit of a $G$-isotypic decomposition of $\mathscr H$, let us consider the spectrum of the Laplacian operator \eqref{eq:operator-L}, understood in this context as an unbounded operator in $L^2(D;V)$. Namely, one has
\[
\sigma(\mathscr L)=\{s_{nm}: n\in \bn, \; m=0,1,2,\dots\},
\]
where $\sqrt{s_{nm}}$ denotes the $n$-th positive zero of the $m$-th Bessel function of the first kind $J_m$. Corresponding to each eigenvalue 
$s_{nm} \in \sigma(\mathscr L)$, there is an associated eigenspace
$\mathscr E_{nm} \subset \mathscr H$ which can be expressed, using standard polar coordinates $(r,\theta)$, as follows
\begin{align*}
  \mathscr E_{nm}:=\left\{J_m(\sqrt{s_{nm}}r)\Big(\cos(m\theta)\vec a+\sin(m\theta)\vec b\Big): \vec a,\, \vec b\in V \right\}.  
\end{align*}
Clearly, one has
\[
\mathscr E_{nm} \simeq \mathcal W_m \otimes V,
\]
for each $(n,m) \in \bn \times \bn \cup \{0\}$ such that $\mathscr H$ admits the $O(2)$-isotypic decomposition
\begin{align*} 
 \mathscr H := \overline{\bigoplus\limits_{m=0}^\infty \mathscr H_m}, \quad \mathscr H_m:= \overline{\bigoplus\limits_{n=1}^\infty \mathscr E_{nm}}, 
\end{align*}
where the closure is taken in $\mathscr H$. In particular, adopting the notations
\begin{equation}\label{eq:enmj}
\mathscr E_{nm}^j :=\left\{J_m(\sqrt{s_{nm}}r)\Big(\cos(m\theta)\vec a+\sin(m\theta)\vec b\Big): \vec a,\, \vec b\in V_j \right\}, 
\end{equation}
and
\begin{align*}
\mathscr A_{n,m}^j(\alpha):= \mathscr A(\alpha)|_{\mathscr E_{nm}^j},
\end{align*}
one has
\[
\mathscr E_{nm}^j \simeq \mathcal W_m \otimes \mathcal V_j^-, \; (n,m) \in \bn \times \bn \cup \{0\}, \; j \in \{1,2,\ldots,r\}.
\]
Hence, $\mathscr H$ admits the $G$-isotypic
\begin{align*} 
   \mathscr H = \bigoplus_{j=1}^r \overline{\bigoplus\limits_{m=0}^\infty \mathscr H_{m,j}}, \quad \mathscr H_{m,j}:=\overline{\bigoplus\limits_{n=1}^\infty \mathscr E_{nm}^j}. 
\end{align*}
To be clear, each $G$-isotypic component $\mathscr H_{m,j}$ is modeled on the irreducible $G$-representation $\mathcal V_{m,j}:= \mathcal W_m \otimes \mathcal V_j^{-}$.
\vs
\begin{lemma} \label{lem:spectral_decomposition}
Under assumption \ref{c0}, each eigenvalue of the linear operator $\mathscr A(\alpha): \mathscr H \rightarrow \mathscr H$ is of the form
\begin{align*}
\xi_{n,m,j}(\alpha) := 1 - \frac{\mu_j(\alpha)}{s_{nm}},
\end{align*}
where $n \in \mathbb{N}$, $m \in \mathbb{N} \cup \{0\}$, $j \in \{1,2,\ldots,r\}$ and $\mu_j(\alpha) \in \sigma(A_j(\alpha))$.
\end{lemma}
\begin{proof}
Indeed, since $\mathscr A(\alpha)$ is $G$-equivariant, one has 
\begin{align*}
   \mathscr A_{n,m}^j(\alpha): \mathscr E_{nm}^j \rightarrow \mathscr E_{nm}^j 
\end{align*}
such that
\begin{align*} 
\sigma( \mathscr A(\alpha)) =  \bigcup\limits_{m=0}^{\infty} \bigcup\limits_{n=1}^{\infty} \bigcup\limits_{j=1}^{r} \sigma(\mathscr A^j_{n,m}(\alpha)), 
\end{align*}
where one easily obtains (see for example \cite{BHKR,book-new})
\begin{align*} 
 \sigma (\mathscr A^j_{n,m}(\alpha)) &= \left\{ 1 - \frac{\mu_j(\alpha)}{s_{nm}} \; : \; \mu_j(\alpha) \in \sigma(A_j(\alpha)) \right\}.   
\end{align*}
\end{proof}
\vs
The product property of the Leray-Schauder $G$-equivariant degree (cf. Appendix \ref{subsec:G-degree}) permits us to express $\gdeg(\mathscr A(\alpha), B(\mathscr H))$, at any regular point $\alpha \in \br$ of equation \eqref{eq:bif}, in terms of a Burnside ring product of the Leray-Schauder $G$-equivariant degrees of the various restrictions $\mathscr A_{n,m}^j(\alpha): \mathscr E^j_{nm} \rightarrow \mathscr E^j_{nm}$ of the $G$-equivariant linear isomorphism $\mathscr A(\alpha): \mathscr H \rightarrow \mathscr H$ to the $G$-subrepresentations $\mathscr E^j_{nm}$ on their respective open unit balls $B(\mathscr  E^j_{nm}):= \{ u \in \mathscr E^j_{nm} \; : \; \| u \|_{\mathscr H} < 1 \}$ as follows
\begin{equation}\label{fact_1}
  \gdeg(\mathscr A(\alpha), B(\mathscr H)) = \prod\limits_{j=1}^r \prod\limits_{m=0}^\infty \prod\limits_{n=1}^\infty \gdeg( \mathscr A_{n,m}^j(\alpha), B(\mathscr E_{nm}^j)). 
\end{equation}
Notice that one has $\gdeg( \mathscr A_{n,m}^j(\alpha), B(\mathscr E_{nm}^j))=(G)$ for almost all indices $m$, $n$ and $j$, so that the product \eqref{fact_1} is well-defined.
Indeed, each Leray-Schauder $G$-equivariant degree $\gdeg( \mathscr A_{n,m}^j(\alpha), B(\mathscr E_{nm}^j))$ is fully specified by the $\mathcal V_j^-$-isotypic multiplicities $\{ m_j \}_{j=1}^r$ together with the real spectra of $\mathscr A(\alpha)$ according to formula,
\begin{equation}\label{fact_2}
\gdeg(\mathscr A_{n,m}^j(\alpha), B( \mathscr E^j_{nm})) =
\begin{cases}
    (\deg_{\mathcal V_{m,j}})^{m_j} \quad & \text{ if } s_{nm} < \mu_j(\alpha); \\
    (G) \quad & \text{ otherwise,}
\end{cases}    
\end{equation}
where $\deg_{\mathcal V_{m,j}} \in A(G)$ is the basic degree (cf. Appendix \ref{subsec:G-degree}) associated with the irreducible $G$-representation $\mathcal V_{m,j}$ and $(G) \in A(G)$ is the unit element of the Burnside 
Ring. In addition, since each basic degree is involutive in the Burnside ring (cf. Appendix \ref{subsec:G-degree}), one has
\begin{equation}\label{fact_3}
  (\deg_{\mathcal V_{m,j}})^{m_j} =
\begin{cases}
    \deg_{\mathcal V_{m,j}} \quad & \text{ if } 2 \nmid m_j; \\
    (G) \quad & \text{ otherwise.}
\end{cases}   
\end{equation}
Putting together \eqref{fact_2} and \eqref{fact_3}, we introduce some notations to keep track of the indices 
\begin{align*}
  \Sigma := \{(n,m,j) \; : \; n \in \mathbb{N}, \; m \in \mathbb{N} \cup \{0\}, \; j \in \{1,2,\ldots,r \} \}, 
\end{align*}
which contribute non-trivially to the Burnside Ring product \eqref{fact_1}. To begin, the \textit{negative} spectrum of $\mathscr A(\alpha): \mathscr H \rightarrow \mathscr H$ is accounted for with the index set
\begin{align}\label{index_1}
   \Sigma_{-}(\alpha) := \left\{ (n,m,j) \in \Sigma \; : \;  1 - \frac{\mu_j(\alpha)}{s_{nm}} < 0 \right\}.
\end{align}
As is well known (see, for example, \cite{Watson}, p. 486), one has
\[
s_{1m} > m(m+2), \; m \geq 0,
\]
from which it follows that the set \eqref{index_1} is finite. 
\vs
Combining \eqref{index_1} with formulas \eqref{fact_1} and \eqref{fact_2}, one obtains
\begin{align*}
    \gdeg(\mathscr A(\alpha), B(\mathscr H)) \; = \prod\limits_{(n,m,j) \in \Sigma_{-}(\alpha)} (\deg_{\mathcal V_{m,j}})^{m_j}.
\end{align*}
 Computation of \eqref{fact_1} can be further reduced by accounting for the even $\mathcal V_j$-isotypic multiplicities $m_j$, whose corresponding basic degrees contribute trivially to the Burnside Product \eqref{fact_1}. We put,
\begin{align}\label{index_2}
\Sigma(\alpha) := \left\{ (n,m,j) \in \Sigma_{-}(\alpha) \; : \; 2 \nmid m_j \right\},
\end{align}
which, together with \eqref{fact_3}, yields
\begin{align} \label{fact_4}
\gdeg(\mathscr A(\alpha), B(\mathscr H)) \; = \prod\limits_{(n,m,j) \in \Sigma(\alpha)} \deg_{\mathcal V_{m,j}}.
\end{align}
\subsection{Computation of the Local Bifurcation Invariant} \label{sec:comp_loc_bif}
Under the assumptions \ref{c1}--\ref{c5}, conditions \ref{b1} and \ref{b2} are satisfied for the operator equation \eqref{eq:bif} (cf. Remarks \ref{rm0}, \ref{rm1}). Therefore, the existence of a branch of non-trivial solutions
to \eqref{eq:Lap} bifurcating from an isolated critical point $(\alpha_0,0) \in \Lambda$ is reduced, by Theorem \ref{th:Kras}, to computation of the local bifurcation invariant $\omega_G(\alpha_0)$. 
\vs
Adopting the notations from Section \ref{sec:bifurcation_abstract}, choose $\alpha^\pm_0 \in (\alpha_0 - \epsilon, \alpha_0 + \epsilon)$ with $\alpha^-_0 \leq \alpha_0 \leq \alpha^+_0$, where $\epsilon > 0$ is chosen such that, for all $0< \vert \alpha - \alpha_0 \vert < \epsilon$, the solution $(\alpha,0) \in M$ is a regular point and put 
$\mathscr A(\alpha^\pm_0) := D_u \mathscr F(\alpha^\pm_0,0)$. Then, the local bifurcation invariant $\omega_G(\alpha_0)$ at the isolated critical point $(\alpha_0,0) \in \br \times \mathscr H$ is given by
\begin{equation}\label{loc_bif_invariant2}
  \omega_G(\alpha_0) = \gdeg(\mathscr A(\alpha^-_0), B(\mathscr H)) - \gdeg(\mathscr A(\alpha^+_0), B(\mathscr H)),  
\end{equation}
where $B(\mathscr  H):= \{ u \in \mathscr H \; : \; \| u \|_{\mathscr H} < 1 \}$ is the open unit ball in $\mathscr H$. Notice that, since $(\alpha^\pm_0,0) \in \br \times \mathscr H$ are regular points of \eqref{eq:bif}, computation of the local bifurcation invariant $\omega_G(\alpha_0)$ amounts to computation of the  Leray-Schuader $G$-equivariant degree of the $G$-equivariant linear isomorphism
\begin{align*}
      \mathscr A(\alpha^\pm_0) := D_u \mathscr F(\alpha^\pm_0,0) = \operatorname{Id} - \mathscr L^{-1}(A(\alpha^\pm_0)): \mathscr H \rightarrow \mathscr H.  
\end{align*}
\vs
\begin{lemma} \label{l4.1}
    Under the assumptions \ref{c0}-\ref{c5} and using the notation \eqref{index_2}, the local bifurcation invariant at an isolated critical point $(\alpha_0,0) \in \Lambda$ with deleted regular neighborhood $\alpha^-_0 \leq \alpha_0 \leq \alpha^+_0$ on which $\mathscr A(\alpha): \mathscr H \rightarrow \mathscr H$ is an isomorphism is given by 
    \begin{equation}\label{eq:loc-bif-Lap}
     \omega_G(\alpha_0) = \prod\limits_{(n,m,j) \in \Sigma(\alpha^-_0)} \deg_{\mathcal V_{m,j}} - \prod\limits_{(n,m,j) \in \Sigma(\alpha^+_0)} \deg_{\mathcal V_{m,j}}.
    \end{equation}
\end{lemma}
In order to formulate the main local equivariant bifurcation result, we must first consider some additional properties of the basic degree. For a more thorough exposition of these topics we refer the reader to \cite{book-new}, \cite{AED}. 
\vs
Take $s \in \mathbb{N}$ and define the \textit{$s$-folding map} as the Lie group homomorphism,
\begin{align*} 
    \psi_s(e^{i \theta},\gamma, \pm 1) = (e^{s \theta}, \gamma, \pm 1), \quad \psi_s(\kappa e^{i \theta},\gamma, \pm 1) = (\kappa e^{is \theta}, \gamma, \pm 1).
\end{align*}
Each $\psi_s: O(2) \times \Gamma \times \mathbb{Z}_2 \rightarrow  O(2) \times \Gamma \times \mathbb{Z}_2$ induces a corresponding Burnside ring homomorphism $\Psi_s: A(G) \rightarrow A(G)$ defined the generators $(H) \in \Phi_0(G)$ by,
\begin{align} 
\label{sf2}
    \Psi_s(H) := ({}^{s}H), \quad {}^{s}H := \psi^{-1}_s(H).
\end{align}
Notice that, for $j \in \{1,\ldots,r\}$ and $m \geq 0$, there is the following relation between basic degrees
\begin{align}
\label{sf3}
    \Psi_s(\deg_{\mathcal V_{m,j}}) = \deg_{\mathcal V_{sm,j}}.
\end{align}
\begin{remark} \label{rm2} \rm
An orbit type which is maximal in $\Phi_0(G;\mathscr H \setminus \{0\})$ is also maximal in $\Phi_0(G;\mathscr H_{0} \setminus \{0\})$. Therefore, any $u \in \mathscr H \setminus \{0\}$ with 
an isotropy $G_u \leq G$ such that $(G_u)$ is maximal in $\Phi_0(G;\mathscr H \setminus \{0\})$ must be \textit{radially symmetric}. In order to detect branches of solutions to \eqref{eq:Lap1} corresponding to maps which are both non-trivial {\it and} non-radial, we must restrict our focus to orbit types which are maximal in $\Phi_0(G;\mathscr H_{m} \setminus \{0\})$ for some positive $m \in \mathbb{N}$. Indeed, to demonstrate the existence of a branch of non-radial solutions bifurcating from some isolated critical point $(\alpha_0,0) \in \Lambda$, it is sufficient to show that for some $m > 0$ there is an orbit type $(H)\in \Phi_0(G;\mathscr H_m\setminus\{0\})$ with $\operatorname{coeff}^H( \omega_G(\alpha_0))\not=0$. In such a case, one can additionally conclude that for some $s\ge 1$ there exists a branch non-radial solutions to \eqref{eq:Lap1} with symmetries at least $(^sH)$ bifurcating from $(\alpha_0,0)$.    
\end{remark}
With motivation from Remark \ref{rm2}, we denote by $\mathfrak M_m$, $m > 0$ the set of all maximal orbit types in $\Phi_0(G;\mathscr H_{m} \setminus \{0\})$ and by 
$\mathfrak M_{m,j}$ the set of orbit types $\Phi_0(G; \mathscr H_{m,j} \setminus \{0\}) \cap \mathfrak M_m$. Since each $(H) \in \mathfrak M_m$ is also an orbit type in $\Phi_0(G; \mathscr H_{m,j} \setminus \{0\})$ for at least one $j \in \{1,2,\ldots, r\}$, one has
\begin{align*}
    \mathfrak M_m = \bigcup_{j=1}^r \mathfrak M_{m,j}.
\end{align*}
Notice that $\Psi_s(\mathfrak M_{m,j}) = \mathfrak M_{sm,j}$ for any $s \in \mathbb{N}$ and $j \in \{1,2,\ldots,r\}$ (in particular, one has $\Psi_s(\mathfrak M_{1,j}) = \mathfrak M_{s,j}$). Hence, any orbit type $(H_0) \in \mathfrak M_{m,j}$ can be recovered from an orbit type in $\mathfrak M_{1,j}$ by the relation $(H):= \Psi_s^{-1}(H_0)$.
\begin{remark} \label{rm3} \rm
Take $m > 0 $ and $j \in \{1,2,\ldots,r\}$. For any basic degree $\deg_{\mathcal V_{m,j}} \in A(G)$ and orbit type $(H) \in \mathfrak M_{1,j}$, the recurrence formula for the Leray-Schauder $G$-equivariant degree (cf. Appendix \ref{subsec:G-degree}) implies
\begin{align*}
    \deg_{\mathcal V_{m,j}} = (G) - y_{j}({}^{m}H) + a_{j}, 
\end{align*}
where $a_{j} \in A(G)$ is such that $\operatorname{coeff}^{{}^{s}H}(a_{j}) = 0$ for all $s \in \bn$ and where the coefficient $y_{j} \in \mathbb{Z}$ is determined by the rule
 \begin{align*} 
    y_{j} =
\begin{cases}
    0 \quad & \text{if } \dim{\mathcal V_{1,j}^{H}} \text{ is even};\\
    1 \quad & \text{if } \dim{\mathcal V_{1,j}^{H}} \text{ is odd and } \vert W(H) \vert = 2; \\
    2 \quad & \text{if } \dim{\mathcal V_{1,j}^{H}} \text{ is odd and } \vert W(H) \vert = 1.
\end{cases}
\end{align*}
equivalently
\begin{align*} 
    y_{j} = \frac{x_0}{2}(1 - (-1)^{\dim{\mathcal V_{1,j}^{H}}}),
\end{align*}
where
 \begin{align*} 
    x_0 =
\begin{cases}
    1 \quad & \text{if } \vert W(H) \vert = 2; \\
    2 \quad & \text{if } \vert W(H) \vert = 1.
\end{cases}
\end{align*}
It follows that non-triviality of the coefficient associated with an orbit type $({}^{m}H) \in \mathfrak M_{m,j}$ in the basic degree $\deg_{\mathcal V_{m,j}}$ is characterized by the parity of $\dim{\mathcal V_{1,j}^{H}}$ in the following way
\begin{align*}
    \operatorname{coeff}^{{}^{m}H}(\deg_{\mathcal V_{m,j}}) \neq 0 \iff 2 \nmid \dim{\mathcal V_{1,j}^{H}}.
\end{align*}
The following result concerns the fate of orbit types belonging to $\mathfrak M_m$ in the Burnside Ring product of basic degrees such as \eqref{fact_4}. In particular, we find that for $m > 0$ and $i,l \in \{1,2,\ldots,r\}$, the coefficient of $(H) \in \mathfrak M_{m,i} \cap \mathfrak M_{m,l}$ is \textit{$2$-nilpotent} with respect to the Burnside Ring product $\deg_{\mathcal V_{m,i}} \cdot \deg_{\mathcal V_{m,l}} \in A(G)$ in the case that both $\dim\mathcal V_{m,i}^{H}$ and $\dim\mathcal V_{m,l}^{H}$ are odd.
\begin{lemma} \label{l4.4}
Take $m > 0 $ and $i,l \in \{1,2,\ldots,r\}$. For 
$(H) \in \mathfrak M_{1,i} \cap \mathfrak M_{1,l}$, one has 
\begin{align*}
    \operatorname{coeff}^{{}^{m}H}(\deg_{\mathcal V_{m,i}} \cdot \deg_{\mathcal V_{m,l}} ) =
        \begin{cases}
                0 \quad & \text{if } \dim\mathcal V_{1,i}^{H} \text{ and } \dim\mathcal V_{1,l}^{H} \text{ are of the same parity;} \\
                - x_0 \quad & \text{else.}
            \end{cases}
    \end{align*}
    equivalently
    \begin{align*}
    \operatorname{coeff}^{{}^{m}H}(\deg_{\mathcal V_{m,i}} \cdot \deg_{\mathcal V_{m,l}} ) = \frac{-x_0}{2} (1 - (-1)^{\dim\mathcal V_{1,i}^{H} + \dim\mathcal V_{1,l}^{H}}).
    \end{align*}
\end{lemma}
\begin{proof}
Consider the Burnside Ring product of the relevant basic degrees
\begin{align*}
    \deg_{\mathcal V_{m,i}} \cdot \deg_{\mathcal V_{m,l}} & = \left( (G) - y_{i}({}^{m}H) + a_{i} \right) \cdot \left( (G) - y_{l}({}^{m}H) + a_{l} \right)\\
    &= (G) - (y_{i} + y_{l} - y_{i}y_{l}\vert W(H)\vert) ({}^{m}H) + a, 
\end{align*}
where $a \in A(G)$ is such that $\operatorname{coeff}^{{}^{s}H}(a) = 0$ for all $s \in \{1,2,\ldots\}$, and put $y_0:= y_{i} + y_{l} - y_{i}y_{l}\vert W(H)\vert$, i.e.
\[
y_0 = \operatorname{coeff}^{{}^{m}H}(\deg_{\mathcal V_{m,i}} \cdot \deg_{\mathcal V_{m,l}} ).
\]
Now, if $\dim\mathcal V_{1,i}^{H}$ and $\dim\mathcal V_{1,l}^{H}$ are both even, then one has $y_{i} = y_{l} = 0$ and the result follows. On the other hand, if $\dim\mathcal V_{1,i}^{H}$ and $\dim\mathcal V_{1,l}^{H}$ are both odd, then one has $y_i = y_l = x_0$ such that
\begin{align*}
    y_0 = x_0(2-x_0 \vert W(H) \vert),
\end{align*}
where in either of the cases, $x_0 = 2$ and $\vert W(H) \vert = 1$ or $x_0 = 1$ and $\vert W(H) \vert = 2$, one has $2 - x_0 \vert W(H) \vert = 0$. If instead one supposes that $\dim\mathcal V_{1,i}^{H}$ and $\dim\mathcal V_{1,l}^{H}$ are of different parities, then one has the two corresponding cases, $y_{i} = x_0$ and $y_{l} = 0$ \textit{or} $y_{i} = 0$ and $y_{l} = x_0$, both of which imply $y_0 = x_0$.
\end{proof}
\vs
Naturally, the Lemma \eqref{l4.4} can be generalized to an orbit type $(H) \in \bigcap_{k=1}^N \mathfrak M_{1,j_k}$ and the Burnside Ring product of the basic degrees $\{\deg_{\mathcal V_{m,j_k}} \}_{k=1}^N$ with $\dim\mathcal V_{1,j_k}^{H}$ odd for an even number of $j_k \in \{1,2,\ldots,r\}$.
\begin{corollary} \label{l4.5}
Take $m > 0, \; N \in \mathbb{N}$ and $j_1,\ldots,j_N \in \{1,2,\ldots,r\}$. For $(H) \in \bigcap_{k=1}^N \mathfrak M_{1,j_k}$, one has
\begin{align*} 
    \operatorname{coeff}^{{}^{m}H} \left( \prod\limits_{k=1}^N \deg_{\mathcal V_{m,j_k}} \right) =
        \begin{cases}
                0 \quad & \text{if } \vert \{ j_k : 2 \nmid \dim\mathcal V_{1,j_k}^{H} \} \vert \text{ is even};\\
                - x_0 \quad & \text{otherwise}
            \end{cases}
    \end{align*}
    equivalently
    \begin{align*} 
    \operatorname{coeff}^{{}^{m}H} \left( \prod\limits_{k=1}^N \deg_{\mathcal V_{m,j_k}} \right) = \frac{-x_0}{2} \left(1 - (-1)^{\sum\limits_{k=1}^N \dim\mathcal V_{1,j_k}^{H}} \right).
    \end{align*}
\end{corollary}
\end{remark}
\vs
Remarks \ref{rm2} and \ref{rm3} permit us to 
further refine our index set \eqref{index_2} as follows. \\ 
Let $(\alpha_0,0) \in \Lambda$ be an isolated critical point with a deleted regular neighborhood $\alpha_0^- < \alpha_0 < \alpha_0^+$ on which $\mathscr A(\alpha): \mathscr H \rightarrow \mathscr H$ is an isomorphism. For a given $s \in \mathbb{N}$, $(H) \in \mathfrak M_{1}$ we put  
\begin{align} \label{index_3}
\Sigma^s(\alpha_0^\pm,H) := 
        \{ (n,m,j) \in \Sigma(\alpha_0^\pm) :
\operatorname{coeff}^{{}^{s}H}(\deg_{\mathcal V_{m,j}}) \neq 0 \}, 
\end{align}
\begin{align} \label{card_1}
    \mathfrak n^s(\alpha_0^\pm, H) := \vert \Sigma^s(\alpha_0^\pm,H) \vert,
\end{align}
and 
\begin{align} \label{card_2}
    \mathfrak m^s(\alpha_0^\pm,H) := \vert  \{ (n,m,j) \in \Sigma (\alpha_0^\pm) : ({}^{s}H) < ({}^{m}H) \text{ and } 2 \nmid \mathfrak n^m(\alpha_0^\pm, H) \} \vert.
\end{align}
\begin{remark} \rm \label{rm4}
Take $(H) \in \mathfrak M_{1}$, $(n,m,j) \in \Sigma$ and $s \in \bn$. If 
$\operatorname{coeff}^{{}^{s}H}(\deg_{\mathcal V_{m,j}}) \neq 0$, then $m = s$. On the other hand, if $({}^{s}H) \leq ({}^{m}H)$, then $s \mid m$. The converse of each statement is not, in general, true.
\end{remark}
In order to keep track of the numbers $s \in \bn$ for which the parities of $\mathfrak n^s(\alpha_0^\pm, H)$ disagree, we put
\begin{equation}\label{eq:loc-indicator}
\mathfrak i^s(\alpha_0,H):=
\begin{cases} 
-1\;&\text{ if } \mathfrak n^s(\alpha_0^-, H) \;\text{ is odd and }  \mathfrak n^s(\alpha_0^+, H) \; \;\text{ is even}  \\
1\;&\text{ if } \mathfrak n^s(\alpha_0^-, H) \;\text{ is even and }  \mathfrak n^s(\alpha_0^+, H) \; \;\text{ is odd},  \\
0 \; & \;\text{ otherwise}.
\end{cases}
\end{equation}
and also
\begin{align} \label{eq:loc-maxfolding}
    \mathfrak s(\alpha_0,H) := \max\{ s \in \bn: \mathfrak i^s(\alpha_0,H) \neq 0 \}.
\end{align}
\begin{remark} \rm \label{rm5}
Notice that $\mathfrak s(\alpha_0,H)$ is well defined and that the numbers $\mathfrak m^{\mathfrak s(\alpha_0,H)}(\alpha_0^\pm,H)$ are of the same parity for any isolated critical point $(\alpha_0,0) \in \Lambda$ and orbit type $(H) \in \mathfrak M_1$. 
\end{remark}
We are finally in a position to formulate our main local bifurcation result.
\begin{theorem}\label{th:bounded}
If $(\alpha_0,0) \in \Lambda$ is an isolated critical point with 
a deleted regular neighborhood $\alpha^-_0 \leq \alpha_0 \leq \alpha^+_0$ on which the linear operator $\mathscr A(\alpha): \mathscr H \rightarrow \mathscr H$ is an isomorphism and if there is an orbit type $(H) \in \mathfrak M_1$ such that $\mathfrak i^{\mathfrak s}(\alpha_0,H)\not=0$ and $\mathfrak i^{m}(\alpha_0,H)=0$ for all $m > \mathfrak s$, where $\mathfrak s:=\mathfrak s(\alpha_0,H) \in \bn$, then one has
\begin{align*}
   \operatorname{coeff}^{{}^{s}H}(\omega_G(\alpha_0)) =  
    \begin{cases}
        (-1)^{^{\mathfrak m^{\mathfrak s}(\alpha^-_0,H)}}\mathfrak i^{\mathfrak s}(\alpha_0,H)x_0 \quad& \text{ if } s = s_0;\\
        0 \quad& \text{ if } s > s_0.
    \end{cases} 
\end{align*}
\end{theorem}
\begin{proof} 
For convenience, take $\alpha^*_0 \in \{\alpha^\pm_0 \}$.  Adopting the notation
\begin{align} \label{loc_bif_inv_3}
\rho_G(\alpha^*_0):= \prod\limits_{(n,m,j) \in \Sigma(\alpha^*_0)} \deg_{\mathcal V_{m,j}},
\end{align}
the local bifurcation invariant (cf. \ref{sec:local-bif-inv}) becomes
\begin{align*}
    \omega_G(\alpha_0) =  \rho_G(\alpha^-_0) - \rho_G(\alpha^+_0).
\end{align*}
If one also puts
\begin{align*}
   \Sigma^0(\alpha_0^*,H) := \{ (n,m,j) \in \Sigma(\alpha_0^*) : \forall_{s' \in \mathbb{N}} \;
\operatorname{coeff}^{{}^{s'}H}(\deg_{\mathcal V_{m,j}}) = 0 \},
\end{align*}
and, for each $s = 0,1,\ldots$ defines 
\begin{align} \label{loc_bif_inv_5}
    \rho^s_G(\alpha^*_0, H) := \prod\limits_{(n,m,j) \in \Sigma^{s}(\alpha^*_0, H)} \deg_{\mathcal V_{m,j}},
\end{align}
then the product of basic degrees \eqref{loc_bif_inv_3} becomes
\begin{align*} 
        \rho_G(\alpha^*_0) = \prod\limits_{s \in \mathbb{N} \cup \{0\}} \rho^s_G(\alpha^*_0, H),
\end{align*}
since $\Sigma(\alpha_0^*) = \bigcup_{s \in \mathbb{N} \cup \{0\}} \Sigma^s(\alpha_0^*,H)$
is a partition of the index set \eqref{index_2}.
It follows, from Lemma \eqref{l4.4} and its Corollary \eqref{l4.5}, that each of \eqref{loc_bif_inv_5} is of the form
\begin{align*}
    \rho^s_G(\alpha^*_0, H) =
    \begin{cases}
        (G) + b_0 \quad& \text{ if } s=0;\\
        (G) - y_s(\alpha^*_0)({}^{s}H) + b_s \quad& \text{ if } s \geq 1,
    \end{cases} 
\end{align*}
where $b_0,b_s \in A(G)$ are such that $
\operatorname{coeff}^{{}^{s'}H}(b_0)= \operatorname{coeff}^{{}^{s'}H}(b_s) = 0$ 
for all $s' \in \mathbb{N}$ and where the coefficients $y_s(\alpha^*_0) \in \mathbb{Z}$ are determined by the rule
\begin{align*}
    y_s(\alpha^*_0) =
    \begin{cases}
        0 \quad& \text{ if } \mathfrak n^s(\alpha^*_0, H) \text{ is even};\\
        x_0 \quad& \text{ if } \mathfrak n^s(\alpha^*_0, H) \text{ is odd},
    \end{cases} 
\end{align*}
equivalently
\begin{align*}
    y_s(\alpha^*_0) = \frac{x_0}{2}(1 - (-1)^{\mathfrak n^s(\alpha^*_0, H)}).
\end{align*} 
Since $\mathfrak n^s(\alpha^\pm_0, H)$ are of the same parity for all $s > \mathfrak s$, one has
\begin{align*}
    \rho_G^s(\alpha^-_0) \cdot \rho_G^s(\alpha^+_0) = (G) + c_s, \; s > \mathfrak s,
\end{align*}
where $c_s \in A(G)$ is such that  $
\operatorname{coeff}^{{}^{s'}H}(c_s)=0$ 
for all $s' \in \mathbb{N}$. Therefore, the local bifurcation invariant can be expressed in terms of the quantities \eqref{loc_bif_inv_5} as follows
\begin{align*}
    \omega_G(\alpha_0) =  \prod\limits_{s > \mathfrak s} \rho_G^{s}(\alpha^-_0,H) \cdot \left( \rho_G^{\mathfrak s}(\alpha^-_0,H) - \rho_G^{\mathfrak s}(\alpha^+_0,H) + \bm \beta \right) 
\end{align*}
where $\bm \beta \in A(G)$ is such that $\operatorname{coeff}^{{}^{s'}H}(\bm \beta) = 0$ for all $s' \geq \mathfrak s$ . 
\vs
To complete the proof of Theorem \ref{th:bounded}, we will need the following Lemma.
\begin{lemma} \label{lem:bounded}
Take $s,s_0 \in \mathbb{N}$. Using the notation \eqref{loc_bif_inv_5}, one has 
\begin{align*}
 \operatorname{coeff}^{{}^{s_0}H}(\rho_G^s(\alpha^*_0,H) \cdot \pm x_0 ({}^{s_0}H)) = 
\begin{cases}
        \mp x_0 \quad& \text{ if } ({}^{s_0}H) \leq ({}^{s}H) \text{ and } \mathfrak n^s(\alpha^*_0, H) \text{ is odd};\\
        \pm x_0 \quad& \text{ otherwise}.
\end{cases} 
\end{align*}
\end{lemma} 
\begin{proof}
Indeed, consider the relevant Burnside Ring product
\begin{align} \label{loc_bif_inv_9}
   \rho_G^s(\alpha^*_0,H) \cdot \pm x_0 ({}^{s_0}H) &= \left( (G) - y_s(\alpha^*_0)({}^{s}H) + b_s \right) \cdot \left( \pm x_0 ({}^{s_0}H) \right) \nonumber \\
   &= \pm x_0({}^{s_0}H) \mp y_s(\alpha^*_0)x_0 ({}^{s}H) \cdot ({}^{s_0}H) + \bm \alpha
\end{align}
where $\bm \alpha \in A(G)$ is such that $\operatorname{coeff}^{{}^{s'}H}(\bm \alpha) = 0$ for all $s' \in \mathbb{N}$. Now, if $\mathfrak n^s(\alpha^*_0, H)$ is even, then $y_s(\alpha^*_0) = 0$ and the result follows. Supposing instead that $2 \nmid \mathfrak n^s(\alpha^*_0, H)$, then \eqref{loc_bif_inv_9} becomes
\begin{align*}
    \rho_G^s(\alpha^*_0,H) \cdot \pm x_0 ({}^{s_0}H) = \pm x_0(1 - x_0 d_0)({}^{s_0}H) + \bm \gamma
\end{align*}
where $\bm \gamma \in A(G)$ is such that $\operatorname{coeff}^{{}^{s'}H}(\bm \gamma) = 0$ for all $s' \in \mathbb{N}$ and $d_0 := \operatorname{coeff}^{{}^{s_0}H}( ({}^{s}H)\cdot({}^{s_0}H))$ is given by the recursive formula (cf. { Appendix} \ref{subsec:G-degree}) as follows
\begin{align*}
    d_0 := \frac{n({}^{s_0}H,{}^{s}H) n({}^{s_0}H,{}^{s_0}H) \vert W(H) \vert^2}{\vert W(H) \vert}.
\end{align*}
Now, if $({}^{s_0}H) \nleq ({}^{s}H)$, then $n({}^{s_0}H,{}^{s}H) = 0$ and the result follows. In the case that $({}^{s_0}H) \leq ({}^{s}H)$, one has $n({}^{s_0}H,{}^{s}H) = n({}^{s_0}H,{}^{s_0}H) = 1$ such that $d_0 = \vert W(H) \vert$  and the result follows from the fact that $x_0 \vert W(H) \vert = 2$.
\end{proof}
\vs \noindent
{\it Completion of the proof of Theorem \ref{th:bounded}} The result follows from Lemma \ref{lem:bounded} together with the observation that
\begin{align*}
\operatorname{coeff}^{{}^{s}H}\left( \rho_G^{\mathfrak s}(\alpha^-_0,H) - \rho_G^{\mathfrak s}(\alpha^+_0,H) + \bm \beta \right) =  
    \begin{cases}
       \mathfrak i^{\mathfrak s}(\alpha_0,H)x_0 \quad& \text{ if } s = \mathfrak s;\\
        0 \quad& \text{ if } s > \mathfrak s.
    \end{cases} 
\end{align*}
\end{proof}
\vs
\begin{corollary}\label{th:cor_bounded}
Under the assumptions \ref{c1}-\ref{c5}, if $(\alpha_0,0) \in \Lambda$ is an isolated critical point with 
a deleted neighborhood $[\alpha^-,\alpha^+] \setminus \{\alpha_0\}$ on which the linear operator $\mathscr A(\alpha): \mathscr H \rightarrow \mathscr H$ is an isomorphism and if there is an orbit type  $(H) \in \mathfrak M_{1}$ and a number $\mathfrak s:=\mathfrak s(\alpha_0,H) \in \bn$ such that $\mathfrak i^{\mathfrak s}(\alpha_0,H)\not=0$ and $\mathfrak i^{m}(\alpha_0,H)=0$ for all $m > \mathfrak s$, then 
system \eqref{eq:bif} admits a branch of non-radial solutions $\mathscr C$ with branching point $(\alpha_0,0)$ and with $(G_u) \geq (^{\mathfrak s}H)$ for all $u \in \mathscr C$.
\end{corollary}

\subsection{Resolution of the Rabinowitz Alternative} \label{sec:res_rab_alt}
Without an appropriate {\it fixed point reduction} of the bifurcation problem \eqref{eq:bif}, we are unable to guarantee that a branch of non-trivial solutions $\mathcal C$ to \eqref{eq:Lap}  bifurcating from a given isolated critical point $(\alpha_0,0) \in \br \times \mathscr H$ whose existence has been established using a Krasnosel'skii type result (e.g. by Theorem \ref{th:bounded}) is not comprised of {\it radial solutions}. 
With this in mind, consider the subgroup
$\bm K:=\{(-1,e,-1),(1,e,1)\}\le O(2)\times \Gamma\times \bz_2$ and denote by $\mathscr F^{\bm K}:\br\times \mathscr H^{\bm K}\to \mathscr H^{\bm K}$ the restriction 
\begin{equation}\label{eq:H-map}
\mathscr F^{\bm K}:=\mathscr F|_{\br\times \mathscr H^{\bm K}},
\end{equation}
of the operator \eqref{eq:Lap1} to the $\bm K$-fixed point space $\mathscr H^{\bm K}$. Clearly, any solution $(\alpha,u) \in \br \times \mathscr H^{\bm K}$ to the equation 
\begin{equation}\label{eq:H-Lap}
\mathscr F^{\bm K}(\alpha,u)=0,
\end{equation}
is also solution to \eqref{eq:bif}. 
Notice also that any radial solution to \eqref{eq:H-Lap} belongs to the set of trivial solutions (cf. Condition \ref{b1})  
\[
M:=\{(\alpha,0): \alpha \in \br, \; 0 \in \mathscr H \}.
\]
Therefore, any branch of of non-trivial solutions to the bifurcation problem \eqref{eq:H-Lap} consists solely of non-radial solutions. In this $\bm K$-fixed point setting, we must 
adapt each of the notions introduced in Section \ref{sec:rab-alt} used to describe the Rabinowitz alternative for the equation \eqref{eq:bif} to the bifurcation map \eqref{eq:H-map}. To begin, notice that the $\bm K$-fixed point space $ \mathscr H^{\bm K}$ is an isometric Hilbert representation of the group $\mathcal G:=N(\bm K)/\bm K=O(2)\times \Gamma$ with the $\mathcal G$-isotypic decomposition 
\begin{align*} 
   \mathscr H^{\bm K} = \bigoplus_{j=1}^r \overline{\bigoplus\limits_{m=1}^\infty \mathscr H_{2m-1,j}}, \quad \mathscr H_{2m-1,j}:=\overline{\bigoplus\limits_{n=1}^\infty \mathscr E_{n,2m-1}^j}. 
\end{align*}
We denote by $\mathscr S^{\bm K}$ the set of $\bm K$-fixed non-trivial solutions to \eqref{eq:bif}, i.e.
\[
\mathscr S^{\bm K} :=\{(\alpha,u)\in \br\times \mathscr H^{\bm K}:\mathscr F^{\bm K}(\alpha,0)=0\;\; \text{ and } \;\; u\not=0\},
\]
and by $\mathscr A^{\bm K} : \br \times \mathscr H^{\bm K}\to \mathscr H^{\bm K}$ the restriction of the operator $\mathscr A: \br \times \mathscr H \rightarrow \mathscr H$ to $\mathscr H^{\bm K}$, i.e. for each $\alpha \in \br$
\begin{equation}\label{H-Linearization}
   \mathscr A^{\bm K}(\alpha): =\mathscr A(\alpha)|_{\mathscr H^{\bm K}}: \mathscr H^{\bm K}\to \mathscr H^{\bm K}. 
\end{equation}
As before, the {\it critical set} of \eqref{eq:H-Lap}, now denoted $\Lambda^{\bm K}$, is the set of trivial solutions $(\alpha_0,0) \in M$ for which $\mathscr A^{\bm K}(\alpha_0)$ is not an isomorphism
\[
\Lambda^{\bm K}:=\{(\alpha,0)\in \br\times \mathscr H: \mathscr A^{\bm K}(\alpha):\mathscr H^{\bm K}\to \mathscr H^{\bm K}\;\; \text{ is not an isomorphism}\}.
\]
Next, we describe the spectrum of \eqref{H-Linearization} in terms of the spectra $\sigma(\mathscr A^j_{n,m}(\alpha))$ (cf. Lemma \ref{lem:spectral_decomposition}) as follows
\[
\sigma(\mathscr A^ {\bm K}(\alpha))=\bigcup_{m=1}^\infty\bigcup_{n=1}^\infty\bigcup_{j=1}^r \sigma(\mathscr A_{n,2m-1}^j(\alpha)).
\]
Assume that for a given $\alpha \in \br$, the operator $\mathscr A^ {\bm K}(\alpha): \mathscr H^{\bm K} \rightarrow \mathscr H^{\bm K}$ is an isomorphism. If we refine the index set \eqref{index_2} to include only those indices $(n,m,j) \in \Sigma(\alpha)$ relevant to the $\bm K$-fixed point setting with the notation
\[
\Sigma^{\bm K}(\alpha) := \{ (n,m,j) \in \Sigma(\alpha) : 2 \nmid m \},
\]
then the $\mathcal G$-equivariant degree $\mathcal G\text{-deg}(\mathscr A^ {\bm K}(\alpha), B(\mathscr H^{\bm K}))$ can be computed as follows
\begin{align*}
    \mathcal G\text{-deg}(\mathscr A^ {\bm K}(\alpha), B(\mathscr H^{\bm K})) = \prod\limits_{(n,m,j) \in \Sigma^{\bm K}(\alpha)} \wt{\deg}_{\cV_{j,m}},
\end{align*}
where, to distinguish between $G$-basic degrees and $\mathcal G$-basic degrees, we have introduced the notation
\[
\wt{\deg}_{\cV_{j,m}}:=\mathcal G\text{-deg}(-\id,B(\cV_{j,m})).
\]
At this point, Lemma \ref{l4.1} can be reformulated for the map 
$\mathscr F^{\bm K}$ as follows:
\begin{lemma}
Under the assumptions \ref{c0}--\ref{c5} and for any isolated critical point $(\alpha_0,0)\in \Lambda^{\bm K}$ with deleted regular neighborhood $\alpha_0^- < \alpha_0 < \alpha_0^+$, the local bifurcation invariant $\omega_{\mathcal G}(\alpha_0) :=\mathcal G\text{-deg}(\mathcal A^{\bm K}(\alpha_0^-), B(\mathscr H^{\bm K}))-\mathcal G\text{-deg}(\mathcal A^{\bm K}(\alpha_0^+), B(\mathscr H^{\bm K}))$ is given by
\begin{align}\label{eq:loc-bif-Lap-K}
\omega_{\mathcal G}(\alpha_0) = \prod\limits_{(n,m,j) \in \Sigma^{\bm K}(\alpha^-)} \wt{\deg}_{\mathcal V_{m,j}} - \prod\limits_{(n,m,j) \in \Sigma^{\bm K}(\alpha^+)}\wt{ \deg}_{\mathcal V_{m,j}}.
\end{align}   
\end{lemma}
Likewise, Theorem \ref{th:Rabinowitz-alt} becomes:
\begin{theorem}\label{th:Rabinowitz-alt-K}{\bf (Rabinowitz' Alternative)}
Under the assumptions \ref{c0}--\ref{c5}, let $\mathcal U \subset \br \times \mathscr H^{\bm K}$  be an open bounded $\mathcal G$-invariant set  with $\partial \mathcal U \cap \Lambda^{\bm K} = \emptyset$. If $\mathcal C$ is a branch of nontrivial solutions to \eqref{eq:H-Lap} bifurcating from an isolated critical critical point $(\alpha_0,0) \in \mathcal U \cap \Lambda$, then one has the following alternative:
\begin{enumerate}[label=$(\alph*)$]
\item  either $\mathcal C \cap \partial \mathcal U \neq \emptyset$;
    \item\label{alt_b} or there exists a finite set
    \begin{align*}
        \mathcal C \cap \Lambda^{\bm K} = \{ (\alpha_0,0),(\alpha_1,0), \ldots, (\alpha_n,0) \},
    \end{align*}
    satisfying the following relation
    \begin{align*}
        \sum\limits_{k=1}^n \omega_{\mathcal G}(\alpha_k) = 0.
    \end{align*}
\end{enumerate} 
\end{theorem}
\noindent To further simplify our exposition, we replace assumption \ref{c0} with:
\begin{enumerate} [label=($\tilde{A}_0$)] 
    \item\label{c0_tilde} For each $j \in \{1,2, \ldots, r\}$ there exists a continuous and bounded map $\mu_j: \br \rightarrow \br$ with
    \[
    A_j(\alpha) := A(\alpha)|_{V_j} = \mu_j(\alpha) \operatorname{Id}|_{V_j}.
    \]
\end{enumerate}
Take $(H) \in \mathfrak M_1$ and $(\alpha_0,0) \in \Lambda^{\bm K}$. Under assumption \ref{c0_tilde}, there is some finite $m' = 0,1, \ldots$ for which $m > m'$ implies $(n,m,j) \notin \Sigma^{\bm K}(\alpha_0)$. Clearly, the quantity
\begin{align} \label{def:maximal_folding_global}
  \bar{\mathfrak s}(H) := \max\limits_{(\alpha,0) \in \Lambda^{\bm K}} \{\mathfrak s(\alpha,H) \}  
\end{align}
and the set
\begin{align} \label{def:set_J}
    \mathfrak J(H) := \{ (\alpha,0) \in \Lambda^{\bm K} : \mathfrak s(\alpha,H) = \bar{\mathfrak s}(H) \}.
\end{align}
are well-defined. The elements of $\mathfrak J(H)$ can always be indexed $(\alpha_1,0), (\alpha_2,0), \ldots \in \mathfrak J(H)$ in such a way that, if $i<j$, then $\alpha_i < \alpha_j$.
\vs
We are now in a position to formulate our main global bifurcation result.
\begin{theorem}\label{th:unbounded-K}
    If there is an orbit type $(H) \in \mathfrak M_1$ for which $2 \nmid \vert \mathfrak J(H) \vert$, then the system \eqref{eq:Lap} admits an unbounded branch $\mathscr C$ of non-radial solutions with symmetries at least $({}^{\bar{\mathfrak s}}H)$, where $\bar{\mathfrak s} := \bar{\mathfrak s}(H)$. Moreover, one has the following alternative: there exists $M>0$ such that, either
\begin{itemize}
\item[(a)] for all $\alpha>M$ with $(\alpha,0)\notin \Lambda^{\bm K}$ one has
$\mathscr C\cap \{\alpha\}\times \mathscr H\not=\emptyset$, or
\item[(b)] for all $\alpha<-M$ with $(\alpha,0)\notin \Lambda^{\bm K}$ one has
$\mathscr C\cap \{\alpha\}\times \mathscr H\not=\emptyset$.
\end{itemize}
\end{theorem}
\begin{remark} \rm \label{rm6}
Conditions (a) and (b) in Theorem \ref{th:unbounded-K} guarantee that the branch $\mathscr C$ of non-radial solutions extends indefinitely either into the direction of increasing $\alpha \to \infty$ or into the decreasing $\alpha \to -\infty$.  
\end{remark}
\begin{proof}
Notice that, for any critical point $(\alpha_i,0) \in \mathfrak J(H)$, the numbers $\mathfrak m^{\bar{\mathfrak s}}(\alpha_i^\pm,H)$ have the same parity (cf. Remark \ref{rm5}). Without loss of generality, assume that $|\mathfrak J(H) | > 1$ and consider any other critical point $(\alpha_j,0) \in \mathfrak J(H)$ with $\alpha_i < \alpha_j$. We will show that $\mathfrak m^{\bar{\mathfrak s}}(\alpha_i^+,H)$ and $\mathfrak m^{\bar{\mathfrak s}}(\alpha_j^-,H)$ are also of the same parity. Indeed, suppose for contradiction that $\mathfrak m^{\bar{\mathfrak s}}(\alpha_i^+,H)$ and $\mathfrak m^{\bar{\mathfrak s}}(\alpha_j^-,H)$ have different parities. Then there is a folding $m > \bar{\mathfrak s}$ such that the sets $\Sigma^{m}(\alpha_i^+,H)$, $\Sigma^{m}(\alpha_j^-,H)$ disagree for some odd numbers of indices (equivalently, the numbers
$\mathfrak n^{m}(\alpha_i^+,H)$, $\mathfrak n^{m}(\alpha_j^-,H)$ have different parity). Consequently,
there must be an intermediate critical point $(\alpha_k,0) \in \mathfrak J(H)$ with $\alpha_i < \alpha_k < \alpha_j$ and $\mathfrak i^{m}(\alpha_k,H) \neq 0$. However, this is in contradiction with the assumption of maximality for $\bar{\mathfrak s}$. It follows that the quantity $(-1)^{\mathfrak m^{\bar{\mathfrak s}}(\alpha_*^-,H)}$ is constant for all critical points $(\alpha_*,0) \in \mathfrak J(H)$.
\vs
On the other hand, let $(\alpha_k,0),(\alpha_{k+1},0) \in \mathfrak J(H)$ be any two consecutive critical points. From the definition of \eqref{def:set_J}, the numbers $\mathfrak i^{\bar{\mathfrak s}}(\alpha_k,H), \mathfrak i^{\bar{\mathfrak s}}(\alpha_{k+1},H)$ must be non-zero. With an argument similar to that the one used above, it can be shown that the numbers $\mathfrak n^{\bar{\mathfrak s}}(\alpha_k^+,H)$ and $\mathfrak n^{\bar{\mathfrak s}}(\alpha_{k+1}^-,H)$ have the same parity, such that
\begin{align*}
    \mathfrak i^{\bar{\mathfrak s}}(\alpha_k,H) \mathfrak i^{\bar{\mathfrak s}}(\alpha_{k+1},H) = -1.
\end{align*}
It follows then, from Theorem \ref{th:bounded}, that $\operatorname{\operatorname{coeff}^{{}^{\bar{\mathfrak s}}H}(\omega_G(\alpha_k))} = - \operatorname{\operatorname{coeff}^{{}^{\bar{\mathfrak s}}H}(\omega_G(\alpha_{k+1}))}$ and also, for any critical point $(\alpha_0,0) \in \Lambda^{\bm K} \setminus \mathfrak J(H)$, that $\operatorname{\operatorname{coeff}^{{}^{\bar{\mathfrak s}}H}(\omega_G(\alpha_0))} = 0$. Therefore, if $2 \nmid \vert \mathfrak J(H) \vert$, there exists an unbounded branch $\mathscr C$ of non-radial solutions bifurcating from each critical point in $\mathfrak J(H)$ with symmetries at least  $(^{\bar{\mathfrak s}}H)$.
Moreover, by Lemma \ref{lem:a-priori}, the branch $\mathscr C$ cannot extend to infinity with respect to the magnitude of vectors $u$ belonging to $\mathscr C$ (for any particular $\alpha$), the only option is that the branch $\mathscr C$ extends to infinity with respect to $\alpha$.
 \end{proof}

\vs
\section{Motivating Example: Vibrating Membranes with Symmetric Coupling}\label{sec:example}
Equations involving the Laplacian operator are sometimes used to describe time-invariant wave or diffusion processes, also called {\it steady-state phenomena}. As a preliminary model for steady-state phenomena, consider the Helmholtz equation
\begin{equation}\label{eq:helmholtz0} 
\begin{cases}
-\Delta u = \lambda^2 u,\\
u|_{\partial D}=0,
\end{cases}
\end{equation}
defined on the planar unit disc $D:= \{z \in \bc : \vert z \vert < 1\}$. Solutions of \eqref{eq:helmholtz0} are called {\it normal modes} and take the form
\begin{align*}
    u_{nm}(r,\theta) = J_m(\sqrt{s_{nm}}r)(A\cos(m\theta)+B\sin(m\theta)), \quad A,B \in \br, \; n \in \bn, \; m=0,1,\ldots
\end{align*}
where each $\sqrt{s_{nm}}$ can be calculated as the $n$-th root of the $m$-th Bessel function of the first kind $J_m$. The classical Helmholtz equation, as is well-known, has limited applicability to real-world problems, which are often inherently nonlinear. For example, in order to study vibrations of a membrane, it is standard to perturb the system \eqref{eq:helmholtz0} with a nonlinearity
$r:\overline D\times \br\to \br$ as follows
\begin{equation}\label{helmholtz}
\begin{cases}
-\Delta u = \lambda^2 u+r(z,u),\\
u|_{\partial D}= 0.
\end{cases}
\end{equation}
\begin{figure}[htbp]
    \centering
    \includegraphics[width=.7\textwidth]{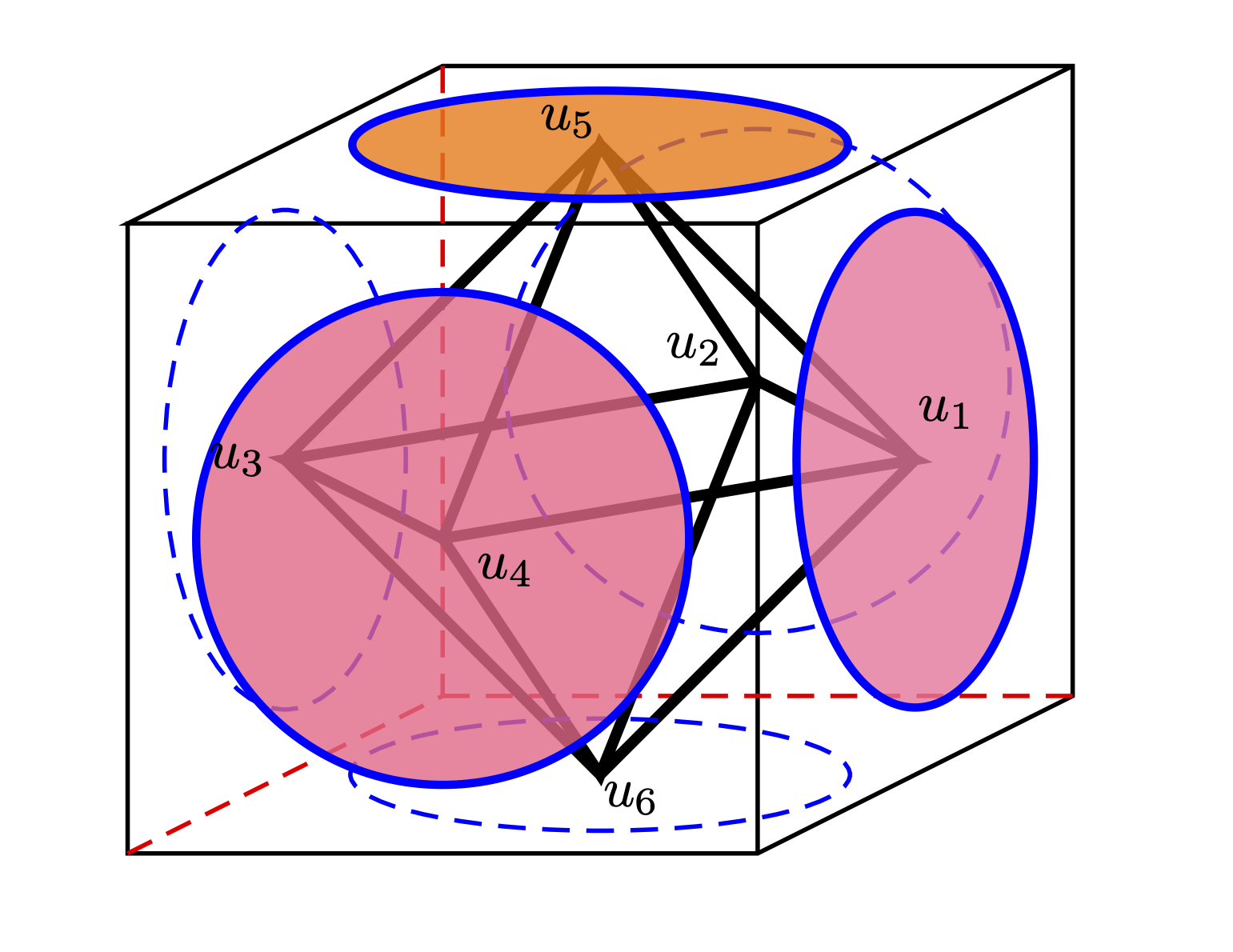}
    \caption{An octahedral configuration of six coupled membranes.}
    \label{fig:6membranes}
\end{figure}
\\Of course, the single membrane model \eqref{helmholtz} is still too simplistic to model
many physically integrated systems. A more realistic problem might instead involve a configuration of some number of coupled oscillating membranes. Consider, for example, a collection of six membranes arranged on corresponding faces of a cube, each coupled to its adjacent neighbors and modeled by a system of identical non-linear Helmholtz equations
\begin{equation}\label{eq:example_Lap}
\begin{cases}
-\Delta u=f(\alpha, z,u), \quad u \in V \\
u|_{\partial D}=0,
\end{cases}
\end{equation}
where $V = \br^6$ and $f: \br \times \overline{D} \times V \rightarrow V$ is the vector-valued function
\begin{equation}\label{eq:ex-f}
f(\alpha,z,u)=\left[\begin{array}{c}  f_1(\alpha,z,u_1,u_2,u_3,u_4,u_5,u_6)\\
f_2(\alpha,z,u_1,u_2,u_3,u_4,u_5,u_6)\\
f_3(\alpha,z,u_1,u_2,u_3,u_4,u_5,u_6)\\
f_4(\alpha,z,u_1,u_2,u_3,u_4,u_5,u_6)\\
f_5(\alpha,z,u_1,u_2,u_3,u_4,u_5,u_6)\\
f_6(\alpha,z,u_1,u_2,u_3,u_4,u_5,u_6)\end{array}\right], \quad 
u = \left[\begin{array}{c}  u_1 \\
 u_2 \\
u_3 \\
u_4 \\
u_5\\
u_6\end{array}\right].
\end{equation}
For compatibility with the results derived in previous sections, we assume that \eqref{eq:ex-f} satisfies the following conditions:
\vs
\begin{enumerate}[label=($E_\arabic*$)]
\item\label{e1} For all $\alpha \in \br$, $z \in D$, $u = (u_1,u_2,u_3,u_4,u_5,u_6)^T \in \br^6$ and $\tau \in \bm O$
\begin{align*}
  \tau f(\alpha, z,u_1,u_2,u_3,u_4,u_5,u_6)=f(\alpha,z,u_{\tau(1)},u_{\tau(2)},u_{\tau(3)},u_{\tau(4)},u_{\tau(5)},u_{\tau(6)}); 
\end{align*}
\item\label{e2} For all $\alpha \in \br$, $z \in D$, $u = (u_1,u_2,u_3,u_4,u_5,u_6)^T \in \br^6$ and $e^{i \theta} \in O(2)$
\begin{align*} 
f(\alpha,e^{i\theta}z,u_1,u_2,u_3,u_4,u_5,u_6)=f(\alpha,z,u_1,u_2,u_3,u_4,u_5,u_6).
\end{align*}
\item\label{e3} For all $\alpha \in \br$, $z \in D$, $u = (u_1,u_2,u_3,u_4,u_5,u_6)^T \in \br^6$, 
\begin{align*}
f(\alpha,z,-u_1,-u_2,-u_3,-u_4,-u_5,-u_6)=-f(\alpha,z,u_1,u_2,u_3,u_4,u_5,u_6).
\end{align*}
\end{enumerate}
\vs
Here, $\bm O \leq S_6$ is used to denote the 
chiral symmetry group of the cube, whose action on $V$ is described by the set of orientation preserving permutations of the corresponding membranes. If we identify $S_4$ with $\bm O$ according to the rule
\begin{gather} 
\label{eq:octahedral_permutation}
(1,2)\;\;\; \leftrightarrow \;\; (1,4)(2,3)(5,6),\quad (1,2,3)\;\;\leftrightarrow \;\; (1,4,6)(3,5,2), \nonumber \\ (1,2)(3,4)\;\; \leftrightarrow \;\; (1,4)(2,3),\quad (1,2,3,4)\;\;\leftrightarrow \;\; (1,2,3,4), \nonumber
\end{gather}
then the faces of the cube (see Figure \ref{fig:6membranes}) are naturally permuted as follows
\begin{align} \label{def:S4-action}
\sigma(u_{1},u_{2},u_{3},u_{4},u_{5},u_6)^T :=
(u_{\sigma(1)},u_{\sigma(2)},u_{\sigma(3)},u_{\sigma(4)},u_{\sigma(5)},u_{\sigma(6)})^T,\quad \forall_{\sigma\in S_4\le S_6},\; u\in V.
\end{align}
In this way, $V$ is an orthogonal $S_4$-representation with respect to the action \eqref{def:S4-action}. The table of characters for $S_4$
\begin{table}[h]
\centering
\begin{tabular}{|c|ccccc|}
\hline
con. classes &$(1)$ & $(1,2)$&$(1,2)(3,4)$ & $(1,2,3)$&$(1,2,3,4)$\\ \hline $\chi_0$ &$1$ & $1$ & $1$ & $1$ &$1$\\
$\chi_1$ &$1$ & $-1$ & $1$ & $1$ &$-1$\\ $\chi_2$ &$2$ & $0$ & $2$ & $-1$ &$0$\\
$\chi_3$ &$3$ & $-1$ & $-1$ & $0$ &$1$\\
$\chi_4$ &$3$ & $1$ & $-1$ & $0$ &$-1$\\
\hline 
$\chi_V$ &$6$ & $0$ & $2$ & $0$ &$2$\\
\hline
\end{tabular}
\caption{Character Table for $S_4$.}
\end{table}
\vs
reveals the relation
 \[ \chi_V=\chi_0+\chi_2+\chi_3, \] 
which can be used to obtain the following $S_4\times \bz_2$-isotypic decomposition of $V$
\begin{align*}
V=\cV_0^-\oplus \cV^-_2\oplus \cV^-_3.
\end{align*}  
Under the assumptions \ref{e1}-\ref{e3}, the function $f$ clearly satisfies conditions \ref{c1}-\ref{c3}. Let us also assume that $f$ satisfies the conditions  \ref{c4} and \ref{c5} with the matrix $A(\alpha):= D_uf(0): V \rightarrow V$ given by
\begin{equation}\label{eq:Aa-ex}
 A(\alpha)= aI + \zeta(\alpha) C,
\end{equation}
where $I:V \rightarrow V$ is the identity matrix, $C:V \rightarrow V$ is the weighted adjacency matrix
\begin{equation}\label{eq:adjacency_matrix-ex}
C =     \left[
    \begin{array}
    [c]{cccccc}
     c &  d & 0 & d & d & d \\
     d &  c & d & 0 & d & d  \\
     0 &  d & c & d & d & d  \\
     d &  0 & d & c & d & d  \\
     d &  d & d & d & c & 0\\
     d &  d & d & d & 0 & c \\
    \end{array}
    \right], 
\end{equation}
defined for some fixed $c,d \in \br$ satisfying the conditions 
\begin{enumerate}[label=($E_\arabic*$),resume]
\item\label{e4} $c > 0$, $d < 0$, $4d+c\ge0$,
\end{enumerate}
and where $\zeta: \br \rightarrow \br$ is the {\it sigmoid function}
\begin{align} \label{eq:sigmoid}
    \zeta(\alpha) = \frac{1}{1 + e^{-\alpha}}.
\end{align}
\begin{figure}[h]
  \centering
  \begin{tikzpicture}[scale=1.5]
    \draw[->] (-3.5, 0) -- (3.5, 0) node[right] {$\alpha$};
    \draw[->] (0, -0.2) -- (0, 1.2);
    \draw[domain=-3:3, smooth, variable=\x, blue, thick] plot ({\x}, {1/(1+exp(-\x))}) node[right] {$\zeta(\alpha) = \frac{1}{1 + e^{-\alpha}}$};
    \draw[dashed, thick] (-3, 0.5) node[left] {$0.5$} -| (0, 0) node[below left] {$0$};
    \draw[dashed, thick] (-3, 1) node[left] {$1.0$} -| (0, 0) node[below left] {$0$};
  \end{tikzpicture}
  \caption{Graph of the sigmoid function $\zeta(\alpha)$.}
  \label{Fig-2}
\end{figure}
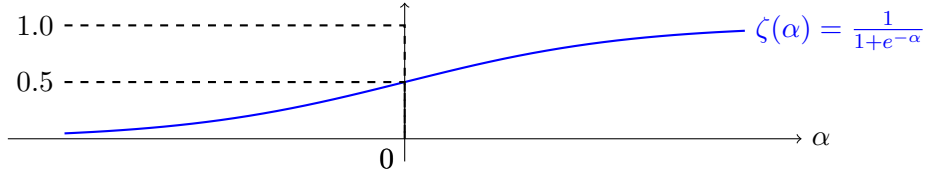
\begin{remark} \rm
The sigmoid function introduces a nonlinear dependence of the coupling in \eqref{eq:example_Lap}
on the bifurcation parameter $\alpha \in \br$. Moreover, the asymptotic behaviour of \eqref{eq:sigmoid} can be understood as the imposition of {\it coupling strength saturation limits,} which might reflect physical constraints such as, for example, bounded ranges for certain material densities.
\end{remark}
With the linearization \eqref{eq:Aa-ex}, system  \eqref{eq:example_Lap} clearly satisfies the condition \ref{c0_tilde}. Indeed, it can be shown that $A(\alpha):V \rightarrow V$ admits three eigenvalues with eigenspaces corresponding to the $S_4$-isotypic components as follows
\begin{align} \label{eq:eigen_values}
\mu_0(\alpha)&:=a+\zeta(\alpha) (c+4d), &E(\mu_0(\alpha))=\cV_0^-, \\
\mu_2(\alpha)&:=a+\zeta(\alpha) (c-2d), &E(\mu_2(\alpha))=\cV_2^-, \nonumber \\
\mu_3(\alpha)&:=a+\zeta(\alpha) c,  &E(\mu_3(\alpha))=\cV_3^-. \nonumber
\end{align}
Notice also that, if $c,d \in \br$ are such that $c+4d, c-2d, c \neq 0$, then the critical set associated with \eqref{eq:ex-f} is discrete. Together with the strict monotonicity of the eigenvalues \eqref{eq:eigen_values} and the distinctness of the numbers $s_{nm}$, it follows that each critical point of \eqref{eq:example_Lap} can be specified by a unique triple $(n,m,j) \in \Sigma$ with the notation
\begin{align} \label{eq:critical_point_id}
    \alpha_{n,m,j}:= \mu_j^{-1}(s_{nm}).
\end{align}
We are now in a position to apply the local equivariant bifurcation results derived in this paper, namely Theorem \ref{th:bounded} and Corollary \ref{th:cor_bounded}, to the example system \eqref{eq:example_Lap}.
\begin{proposition} \label{prop:example_local}
Under assumptions \ref{e1}--\ref{e4}, every critical point $\alpha_{n,m,j}$ of \eqref{eq:example_Lap} is a branching point for a branch of non-trivial solutions with symmetries at least $({}^{s}H)$ for some orbit type of maximal kind $(H) \in \mathfrak 
 M_{1,j}$ and some folding $s \geq m$.
\end{proposition}
\begin{proof}
Let $\alpha_{n,m,j}$ be a critical point of \eqref{eq:example_Lap} with a deleted regular neighborhood $\alpha_{n,m,j}^- < \alpha_{n,m,j} < \alpha_{n,m,j}^+$ on which $\mathscr A(\alpha)$ is an isomorphism and choose any orbit type of maximal kind $(H) \in \mathfrak M_{1,j}$ with $2 \nmid \dim \mathcal V_{1,j}^H$. The unique identification of $(n,m,j) \in \Sigma$ with $\alpha_{n,m,j}$ (cf. \eqref{eq:critical_point_id}) together with the strict monotonicity of the eigenvalues \eqref{eq:eigen_values} imply the set difference $\Sigma(\alpha_{n,m,j}^+) \setminus \Sigma(\alpha_{n,m,j}^-) = \{(n,m,j) \}$. Therefore, the numbers $\mathfrak n^s(\alpha_{n,m,j}^\pm,H)$ agree for all $s \neq m$ and are consecutive in the case of $s = m$ and the result follows from Corollary \eqref{th:cor_bounded}.
\end{proof}
\vs
In order to demonstrate how Theorem \ref{th:unbounded-K}, the main global equivariant bifurcation result of this paper, can be applied to a system of the form \eqref{eq:example_Lap}, we introduce the following simplifying assumption:
\begin{enumerate}[label=($E_\arabic*$),resume]
\item\label{e5} there exists an isotypic index $j \in \{0,2,3\}$, an odd number $m_0 \in 2 \bn - 1$ and a pair of numbers $n_l,n_u \in \bn$ with $n_l \leq n_u$ and $2 \mid n_u - n_l$ such that the Bessel root $s_{nm}$ lies:
\begin{enumerate}[label=(\roman*)]
    \item within the codomain of $\mu_j$ if $m = m_0$ and $n_l \leq n \leq n_u$;
    \item outside the codomain of $\mu_j$ if either, $m = m_0$ and $n < n_l$ or $n > n_u$, or if $m > m_0$ for any $n \in \bn$.
\end{enumerate}
\end{enumerate}
Condition \ref{e5} guarantees that an odd number of critical points $(\alpha,0) \in \Lambda^{\bm K}$ will obtain $\mathfrak s(\alpha,H) = \bar{\mathfrak s}(H)$ for every orbit type of maximal kind $(H) \in \mathfrak M_{1,j}$ with $2 \nmid \dim \mathcal V_{1,j}^H$.
\begin{proposition} \label{prop:example_global}
    Under assumptions \ref{e1}--\ref{e5}, the trivial solution of \eqref{eq:example_Lap} undergoes a global bifurcation of non-radial solutions with symmetries at least $({}^{s}H)$ at every critical point $\alpha_{n,m_0,j}$ with $n_l \leq n \leq n_u$ for every orbit type of maximal kind $(H) \in \mathfrak M_{1,j}$ with $2 \nmid \dim \mathcal V_{1,j}^H$.
\end{proposition}
\begin{proof}
    Choose $(H) \in \mathfrak M_{1,j}$ with $2 \nmid \dim \mathcal V_{1,j}^H$ and notice that, since $\bar {\mathfrak s}(H) = m_0$ and $\alpha_{n,m_0,j} \in \mathfrak J(H)$ for every $n_l \leq n \leq n_u$, one has $2 \nmid | \mathfrak J(H) |$ and the result follows from Theorem \ref{th:unbounded-K}.
\end{proof}
\vs
So that Propositions \ref{prop:example_local} and \ref{prop:example_global} can be verified empirically,
let us specify a particular bifurcation problem associated with system \eqref{eq:example_Lap} by choosing the following values for the constants
$a,c,d \in \br$ 
\begin{align}\label{def:constant_assignment}
    a = 32, \; c = 5, \; d = -1.
\end{align}
For convenience, we include graphs of the three eigenvalues 
\eqref{eq:eigen_values} with the assignments \eqref{def:constant_assignment} in Figure
\ref{fig:eigen_functions} and the numbers $s_{nm}$ for $m = 0, \ldots, 10$ and $n = 1, \ldots, 9$ in Table \ref{tab:bessel_zeros}.
\begin{table}[htbp] \label{tab:bessel_zeros}
\centering
\resizebox{0.95\columnwidth}{!}{
\scalebox{.67}{\tiny\begin{tabular}{ 
|p{0.6cm}||p{0.85cm}|p{0.85cm}|p{0.85cm}|p{0.85cm}|p{0.85cm}|p{0.85cm}|p{0.85cm}|p{0.85cm}|p{0.85cm}|}
\hline
  & $n=1$ & $n=2$ & $n=3$ & $n=4$ & $n=5$ & $n=6$ & $n=7$ & $n=8$ & $n=9$  \\ \hline 
m=0 & 5.783 & 30.471 & 74.887 & 139.04 &222.932 & 326.563 & 449.934 & 593.043 & 755.891 
\\ \hline
m=1 &14.682 & 49.218 & 103.499 & 177.521 & 271.282 & 384.782 & 518.021 & 671 & 843.718 \\ \hline
m=2 & 26.374 & 70.85 & 135.021 & 218.92 & 322.555 & 445.928 & 589.038 & 751.888 & 934.476  \\ \hline
m=3 & 40.706 & 95.278 & 169.395 & 263.201 & 376.625 & 509.98 & 662.968 & 835.693 & 1028.15 \\ \hline
m=4 & 57.583 & 122.428 & 206.57 &310.322 & 433.761 & 576.913 & 739.79 & 922.398 & 1124.74 \\ \hline
m=5 & 76.939 & 152.241 & 246.495 & 360.245 & 493.631 & 646.702 & 819.483 & 1011.99 & 1224.21 \\ \hline
m=6 & 98.726 & 184.67 & 289.13 & 412.934 & 556.303 & 719.321 & 902.024 & 1104.44 & 1326.56 \\ \hline
m=7 & 122.908 & 219.67 & 334.436&468.356 & 621.751 & 794.743 & 987.392 & 1199.73 & 1431.77  \\ \hline
m=8 & 149.453 & 257.21 & 382.38 & 526.481 & 689.946 & 872.946 & 1075.56 & 1297.84 & 1539.81   \\ \hline
m=9 & 178.337 & 297.26 & 432.933 & 587.281 & 760.863 & 953.907 & 1166.52 & 1398.77 & 1650.68 \\ \hline
m=10 & 209.54 & 339.793 & 486.07 & 650.732 & 834.48 & 1037.6 & 1260.24 & 1502.48 & 1764.35  \\ \hline
\end{tabular}}
}
\caption{Squared Zeros of Bessel functions $s_{nm}$ for $0 \le m\le 10$, $1\le n \le 9$.}
\end{table}
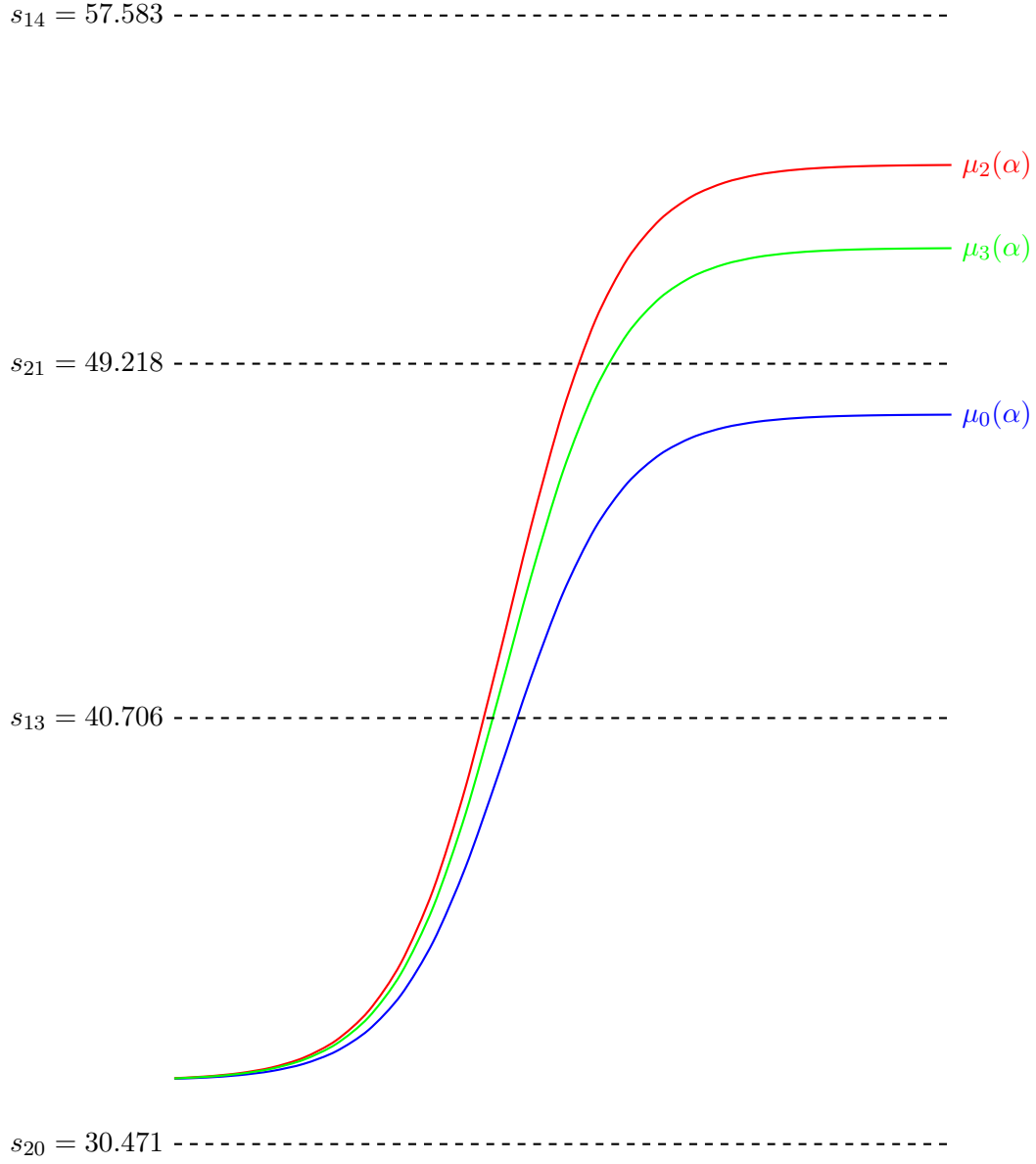
\begin{figure}[h]
  \centering
  \begin{tikzpicture}[scale=.75,yscale=.75]
    \draw[domain=-6:8, smooth, variable=\x, blue, thick] plot ({\x}, {32 + 16/(1+exp(-1*\x))}) node[right] {$\mu_0(\alpha)$};
    \draw[domain=-6:8, smooth, variable=\x, red, thick] plot ({\x}, {32 + 22/(1+exp(-1*\x))}) node[right] {$\mu_2(\alpha)$};
    \draw[domain=-6:8, smooth, variable=\x, green, thick] plot ({\x}, {32 + 20/(1+exp(-1*\x))}) node[right] {$\mu_3(\alpha)$};
    \draw[dashed, thick] (-6, 57.583) node[left] {$s_{14} = 57.583$} -| (8, 57.583) ;
    \draw[dashed, thick] (-6, 40.706) node[left] {$s_{13} = 40.706$} -| (8, 40.706) ;
    \draw[dashed, thick] (-6, 49.218) node[left] {$s_{21} = 49.218$} -| (8, 49.218) ;
    \draw[dashed, thick] (-6, 30.471) node[left] {$s_{20} = 30.471$} -| (8, 30.471) ;
  \end{tikzpicture}
  \caption{Graph of the eigenvalues $\mu_0(\alpha),\mu_2(\alpha)$ and $\mu_3(\alpha)$ and the consecutive numbers $s_{20} < s_{13} < s_{21} < s_{14}$.}
  \label{fig:eigen_functions}
\end{figure}
Notice that the critical set for \eqref{eq:example_Lap} consists of exactly five critical points that can be distinguished, using the notation \eqref{eq:critical_point_id}, as follows
\begin{align} \label{eq:ex_critical_set}
    \Lambda = \{ (\alpha_{1,3,0},0),(\alpha_{1,3,2},0),(\alpha_{1,3,3},0),(\alpha_{2,1,2},0),(\alpha_{2,1,3},0) \}.
\end{align}
In order to employ the results presented in previous sections to our bifurcation problem, we must first identify the maximal orbit types in $\Phi_0(G; \mathscr H_{1}\setminus \{0 \})$. The following G.A.P code can be used to generate a complete list of the sets $\mathfrak M_{1,j}$ and the corresponding Basic Degrees $\deg_{\mathcal V_{1,j}}$ for $j = 1,2,3$.
\begin{lstlisting}  [language=GAP, frame=single] 
# A G.A.P Program for the computation of Maximal Orbit Types and Basic Degrees associated with the G-isotypic decomposition   V = V_0 \times V_2 \times V_3
LoadPackage ("EquiDeg"); 
# create groups O(2)xS4xZ2
o2:=OrthogonalGroupOverReal(2);
s4:=SymmetricGroup(4);
z2:=pCyclicGroup(2);
# generate S4xZ2
g1:=DirectProduct(s4,z2);
# set names for subgroup conjugacy classes in S4xZ2
SetCCSsAbbrv(g1, [ "Z1", "Z2", "D1z","D1","Z2m", "Z1p", 
"Z3", "V4", "D2z", "Z4", "D2", "D1p","D2d", 
"V4m", "D2p", "Z4d", "D3", "D3z", "Z3p",
"V4p", "D4z", "D4d", "Z4p", "D4", "D4z",
 "D4hd", "D3p", "A4", "D4p", "A4p",  "S4", "S4m", "S4p"]);
# generate O(2)xS4xZ2
G:=DirectProduct(o2,g1);
ccs:=ConjugacyClassesSubgroups(G);
# find Maximal Orbit Types and Basic Degrees in first O(2)-isotypic component
irrs := Irr(G); 
# The G-istoypic components 0,2,3 are indexed in G.A.P as 3,5,7
for i in [ 3,5,7 ] do
    max_orbtyps := MaximalOrbitTypes( irrs[1,i] );
    basic_deg := BasicDegree( irrs[1,i] );
    PrintFormatted( "\n Basic Degree associated with irrep V_{1,j} where j= \n", i );
    View(basic_deg);
    PrintFormatted( "\n Maximal Orbit Types in M_1,{} \n", i );
    View(max_orbtyps);
od;
Print( "Done!\n" );
\end{lstlisting}
For any $m > 0$ and $j \in \{0,2,3\}$, the output of the above program can be used to describe the set of maximal orbit types $\mathfrak M_{m,j} \subset \Phi_0(G)$ the basic degrees $\deg_{\mathcal V_{m,j}} \in A(G)$ (cf. \eqref{sf2} and \eqref{sf3}), using amalgamated notation (cf. Appendix \ref{sec:amalgamated_notation}), as follows:
\begin{align*}
\mathfrak M_{m,0}=&\Big\{(D_{2m}^{D_m}\times^{S_4}S_4^p) \Big\},\\
\mathfrak M_{m,2}=&\Big\{
(D_{6m}^{\bz_m}\times ^{V_4}S_4^p),(D_{2m}^{D_m}\times ^{D_4}D_4^p), (D_{2m}^{D_m}\times ^{D_4^{ \hat{d}}}D_4^p)\Big\},\\
\mathfrak M_{m,3}=&\Big\{(D_{2m}^{D_m}\times ^{D_4^z}D_4^p),
(D_{2m}^{D_m}\times ^{D_3^z}D_3^p),(D_{2m}^{D_m}\times ^{D_2^d}D_4^z),(D_{4m}^{\bz_m}\times^{\bz_2^-}D_4^p),
(D_{6m}^{\bz_m}\times D_3^p)\Big\}, 
\end{align*}
\begin{align*}
\deg_{\cV_{m,0}}&=(G)-(D_{2m}^{D_m}\times^{S_4}S_4^p),\\
\deg_{\cV_{m,2}}&=(G)-2(D_{6m}^{\bz_m}\times ^{V_4}S_4^p)-(D_{2m}^{D_m}\times ^{D_4}D_4^p)- (D_{2m}^{D_m}\times ^{D_4^{ \hat{d}}}D_4^p)\\
&+(D_{2m}^{D_m}\times ^{V_4}V_4^p)+4(D_{2m}^{\bz_m}\times ^{V_4}D_4^p),\\
\deg_{\cV_{m,3}}&=(G)-(D_{2m}^{D_m}\times ^{D_4^z}D_4^p)
-(D_{2m}^{D_m}\times ^{D_3^z}D_3^p)-(D_{2m}^{D_m}\times^{D_2^d}D_4^z)\\
&-2(D_{4m}^{\bz_m}\times^{\bz_2^-}D_4^p)-2(D_{6m}^{\bz_m}\times D_3^p)+(D_{2m}^{\bz_m}\times^{\bz_2^-}D_2^p)\\
&+2(D_{2m}^{\bz_m}\times ^{\bz_2^-}V_4^p)+2(D_{2m}^{\bz_m}\times ^{\bz_2^-}D_4^z)+2(D_{2m}^{\bz_m}\times ^{D_1^z}D_4^z)-(D_{2m}^{\bz_m}\times \bz_1^p)-2(D_{2m}^{\bz_m}\times D_2^p).
\end{align*}
We can deduce, from the coefficients of the Basic Degrees $\deg_{\mathcal V_{1,j}} \in A(G)$ (cf. Remark \ref{rm3}), that the fixed point space $\mathcal V_{1,j}^H$ has odd dimension for any $j \in \{0,2,3 \}$ and $(H) \in \mathfrak M_{1,j}$.
Since the critical set \eqref{eq:ex_critical_set} is manageably small, we also can manually compute the cardinalities \eqref{card_1} for each of the critical points $(\alpha_{n,m,j},0) \in \Lambda$ to obtain the quantities
\begin{align*}
    \begin{cases}
       \mathfrak s(\alpha_{1,3,0},H) = 3 \quad & \text{ if } (H) \in \mathfrak M_{1,0}; \\
         \mathfrak s(\alpha_{1,3,2},H) = 3, \; \mathfrak s(\alpha_{2,1,2},H)=1 \quad & \text{ if } (H) \in \mathfrak M_{1,2}; \\
          \mathfrak s(\alpha_{1,3,3},H) = 3, \; \mathfrak s(\alpha_{2,1,3},H)=1 \quad & \text{ if } (H) \in \mathfrak M_{1,3},
    \end{cases}
\end{align*}
and
\begin{align*}
    \begin{cases}
        \mathfrak i^3(\alpha_{1,3,0},H) = 1 \quad & \text{ if } (H) \in \mathfrak M_{1,0}; \\
         \mathfrak i^3(\alpha_{1,3,2},H) = i^1(\alpha_{2,1,2},H) = 1 \quad & \text{ if } (H) \in \mathfrak M_{1,2}; \\
         \mathfrak i^3(\alpha_{1,3,3},H) = i^1(\alpha_{2,1,3},H) = -1 \quad & \text{ if } (H) \in \mathfrak M_{1,3}.
    \end{cases}
\end{align*}
Therefore, from Theorem \ref{th:bounded}, it follows that every critical point $(\alpha_{n,m,j},0) \in \Lambda$ is a branching point for branches of non-trivial solutions to the problem \eqref{eq:example_Lap} with symmetries at least $({}^{3}H)$ in the case of $(\alpha_{1,3,2},0),(\alpha_{1,3,3},0),(\alpha_{1,3,0},0) \in \Lambda$ and with symmetries at least $(H)$ in the case of $(\alpha_{2,1,2},0),(\alpha_{2,1,3},0) \in \Lambda$ for each orbit type of maximal kind $(H) \in \mathfrak M_{1,j}$. The following G.A.P code can be used to computationally verify the non-triviality of the local bifurcation invariant (specifically, the non-triviality of the relevant coefficients) at each of the critical points in terms of the basic degrees 
\[
\deg_{\mathcal V_{3,2}}, \deg_{\mathcal V_{3,3}},\deg_{\mathcal V_{3,0}}, \deg_{\mathcal V_{1,2}}, \deg_{\mathcal V_{1,3}} \in A(G),
\]
and the unit Burnside Ring element $(G)$ according to the rule derived in Lemma \eqref{l4.1}
\begin{align*}
    \omega_G(\alpha_{1,3,2}) &= (G) - \deg_{\mathcal V_{3,2}},\\
    \omega_G(\alpha_{1,3,3}) &= \deg_{\mathcal V_{3,2}} - \deg_{\mathcal V_{3,2}} \cdot \deg_{\mathcal V_{3,3}}, \\
    \omega_G(\alpha_{1,3,0}) &= \deg_{\mathcal V_{3,2}} \cdot \deg_{\mathcal V_{3,3}} -  \deg_{\mathcal V_{3,2}} \cdot \deg_{\mathcal V_{3,3}} \cdot \deg_{\mathcal V_{3,0}},\\
    \omega_G(\alpha_{2,1,2}) &= \deg_{\mathcal V_{3,2}} \cdot \deg_{\mathcal V_{3,3}} \cdot \deg_{\mathcal V_{3,0}} - \deg_{\mathcal V_{3,2}} \cdot \deg_{\mathcal V_{3,3}} \cdot \deg_{\mathcal V_{3,0}} \cdot \deg_{\mathcal V_{1,2}},\\
\omega_G(\alpha_{2,1,3}) &= \deg_{\mathcal V_{3,2}} \cdot \deg_{\mathcal V_{3,3}} \cdot \deg_{\mathcal V_{3,0}} \cdot \deg_{\mathcal V_{1,2}} - \deg_{\mathcal V_{3,2}} \cdot \deg_{\mathcal V_{3,3}} \cdot \deg_{\mathcal V_{3,0}} \cdot \deg_{\mathcal V_{1,2}} \cdot \deg_{\mathcal V_{1,3}}.
\end{align*}
\begin{lstlisting}  [language=GAP, frame=single,belowskip=0pt] 
# A G.A.P Program for the Burnside Ring product  of Basic Degrees to be used for the verification of non-triviality of local bifurcation invariants
# Initialize the Burnside Ring A(G), with unit element U =(G)
AG:=BurnsideRing(G); H:=Basis(AG); U:=H[0,131];
# Initialize the relevant Basic Degrees
d_30 := BasicDegree(Irr(G)[3,3]);
d_32 := BasicDegree(Irr(G)[3,5]);
d_33 := BasicDegree(Irr(G)[3,7]);
d_12 := BasicDegree(Irr(G)[1,5]);
d_13 := BasicDegree(Irr(G)[1,7]);
# Compute Products of Basic Degrees Cumulatively
p_1 := d_32; p_2 := p_1*d_33; p_3 := p_2*d_30; 
p_4 := p_3*d_12; p_5 := p_4*d_13;
# Compute the local bifurcation invariants w_{n,m,j}
w_132 := U - p_1; w_133 := p_1 - p_2; w_130 := p_2 - p_3; 
w_212 := p_3 - p_4; w_213 := p_4 - p_5;
# Sum the local bifurcation invariants
sum := w_132 + w_133 + w_130 + w_212 + w_213
\end{lstlisting}
The output of the above program can be expressed using amalgamated notation as follows: 
\begin{align*}
\omega_G(\alpha_{1,3,2}) &= -4(D_6^{\bz_3} \times^{V_4}D_{4}^p) - (D_6^{D_3} \times^{V_4}V_{4}^p) + (D_6^{D_3} \times^{D_4}D_{4}^p) + (D_6^{D_3} \times^{D_4^{\hat{d}}}D_{4}^p) \\ & \quad + 2(D_{18}^{\bz_3} \times^{V_4}S_{4}^p), \\
    \omega_G(\alpha_{1,3,3}) &= 4(D_6^{\bz_3} \times^{\bz_1}D_{1}^p) + 2(D_6^{\bz_3} \times^{\bz_1}D_{2}^p) + (D_6^{D_3} \times^{\bz_1}\bz_{1}^p)  -2(D_6^{\bz_3} \times^{D_1}D_{4}^z) \\ & \quad - 2(D_6^{\bz_3} \times^{\bz_2^-}D_{4}^z) - 2(D_6^{\bz_3} \times^{\bz_2^-}V_{4}^p) + 2(D_6^{\bz_3} \times^{\bz_2}\bz_{4}^p) + 2(D_6^{\bz_3} \times^{\bz_2}D_{4}^z) \\ & \quad +2(D_6^{\bz_3} \times^{\bz_2}V_{4}^p) - (D_6^{D_3} \times^{D_1}D_1^p)
    -(D_6^{D_3} \times^{D_1^z}D_{1}^p) - (D_6^{D_3} \times^{\bz_2^-}D_{2}^p) \\ & \quad + (D_6^{D_3} \times^{\bz_2}D_2^p) - 2(D_{18}^{\bz_3} \times^{\bz_1}D_3^p) + 2(D_{12}^{\bz_3} \times^{\bz_2^-}D_4^p)
    - 2(D_{6}^{\bz_3} \times^{\bz_4}D_4^p) \\ & \quad - 2(D_{6}^{\bz_3} \times^{D_2^z}D_4^p) + (D_{6}^{D_3} \times^{D_2^d}D_4^z) - (D_{6}^{D_3} \times^{\bz_4}\bz_4^p) -(D_{6}^{D_3} \times^{D_2^z}D_4^z) 
    \\ & \quad
     + (D_{6}^{D_3} \times^{D_3^z}D_3^p) + (D_{6}^{D_3} \times^{D_4^z}D_4^p), \\
    \omega_G(\alpha_{1,3,0}) &= 2(D_{6}^{\bz_3} \times^{D_1}D_4^z) - 2(D_{6}^{\bz_3} \times^{\bz_2}\bz_4^p) - 2(D_{6}^{\bz_3} \times^{\bz_2}D_4^z) - 2(D_{6}^{\bz_3} \times^{\bz_2}V_4^p) \\ & \quad + (D_{6}^{D_3} \times^{D_1}D_1^p)
    - (D_{6}^{D_3} \times^{\bz_2}D_2^p) 
    - 2(D_{6}^{\bz_3} \times^{\bz_3}D_3^p) - (D_{6}^{D_3} \times^{\bz_3}\bz_3^p) 
    \\ & \quad + 2(D_{6}^{\bz_3} \times^{\bz_4}D_4^p) + 3(D_{6}^{\bz_3} \times^{V_4}D_4^p) 
    + (D_{6}^{D_3} \times^{\bz_4}\bz_4^p) + (D_{6}^{D_3} \times^{V_4}V_4^p) \\ & \quad - 2(D_{6}^{D_3} \times^{D_4}D_4^p) + (D_{6}^{D_3} \times^{S_4}S_4^p)\\
\omega_G(\alpha_{2,1,2}) &=  2(D_2^{\bz_1} \times^{\bz_1}D_2^p) + (D_2^{D_1} \times^{\bz_1}\bz_1^p) - 2 (D_2^{\bz_1} \times^{D_1}D_4^z)
-2 (D_2^{\bz_1} \times^{D_1^z}D_4^z) \\ & \quad 
- (D_2^{D_1} \times^{D_1}D_1^p)
- (D_2^{D_1} \times^{D_1^z}D_1^p)
+ 4 (D_6^{\bz_1} \times^{\bz_1}D_3^p)
- 2(D_2^{\bz_1} \times^{\bz_4}D_4^p)
\\ & \quad 
+ 2(D_2^{\bz_1} \times^{D_2^z}D_4^p)
+ 2(D_2^{\bz_1} \times^{V_4}D_4^p)
- (D_2^{D_1} \times^{\bz_4}\bz_4^p)
+ (D_2^{D_1} \times^{D_2^z}D_4^z) \\ & \quad 
+ (D_2^{D_1} \times^{D_4}D_4^p)
- (D_2^{D_1} \times^{D_4^{\hat{d}}}D_4^p)
- 2(D_6^{\bz_1} \times^{V_4}S_4^p)
\\
\omega_G(\alpha_{2,1,3}) &= - 4(D_2^{\bz_1} \times^{\bz_1}D_1^p)
- 4(D_2^{\bz_1} \times^{\bz_1}D_2^p)
- 2(D_2^{D_1} \times^{\bz_1}\bz_1^p)
+ 2(D_2^{\bz_1} \times^{D_1}D_4^z) \\ & \quad 
+ 2(D_2^{\bz_1} \times^{D_1^z}D_4^z)
+ 2(D_2^{\bz_1} \times^{\bz_2^-}D_4^z)
+ 2(D_2^{\bz_1} \times^{\bz_2^-}V_4^p)
+ (D_2^{D_1} \times^{D_1}D_1^p) \\ & \quad 
+ 2(D_2^{D_1} \times^{D_1^z}D_1^p)
+ (D_2^{D_1} \times^{\bz_2^-}D_2^p)
- 2(D_6^{\bz_1} \times^{\bz_1}D_3^p)
+ 2(D_2^{\bz_1} \times^{\bz_3}D_3^p) \\ & \quad 
+ (D_2^{D_1} \times^{\bz_3}\bz_3^p)
- 2(D_4^{\bz_1} \times^{\bz_2^-}D_4^p)
+ 2(D_2^{\bz_1} \times^{\bz_4}D_4^p) 
- (D_2^{D_1} \times^{D_2^d}D_4^z) \\ & \quad 
+ (D_2^{D_1} \times^{\bz_4}\bz_4^p)
+ (D_2^{D_1} \times^{D_3^z}D_3^p)
- (D_2^{D_1} \times^{D_4^z}D_4^p).
\end{align*}
Next, we can determine the global behaviour of some of the branches predicted using Theorem \eqref{th:bounded} by establishing the parity of the set \eqref{def:set_J}. Notice that, for every $j \in \{0,2,3\}$ and $(H) \in \mathfrak M_{1,j}$, the quantity $\bar {\mathfrak s}(H) = 3$ is obtained for exactly one critical point $(\alpha_{1,3,j},0)$, i.e.
\begin{align*}
     \mathfrak J(H) =
    \begin{cases}
        \{ (\alpha_{1,3,0},0) \} \quad & \text{ if } (H) \in \mathfrak M_{1,0}; \\
         \{ (\alpha_{1,3,2},0) \} \quad & \text{ if } (H) \in \mathfrak M_{1,2}; \\
         \{ (\alpha_{1,3,3},0) \} \quad & \text{ if } (H) \in \mathfrak M_{1,3}.
    \end{cases}
\end{align*}
Therefore,
any branch of non-trivial solutions with symmetries at least $({}^{3}H)$ bifurcating from $(\alpha_{1,3,j},0)$, for every $j \in \{0,2,3\}$ and $(H) \in \mathfrak M_{1,j}$, consists only of non-radial solutions and is unbounded. Of course, since the critical set \eqref{eq:ex_critical_set} is finite, one needs only sum the relevant local bifurcation invariants to arrive at the same conclusion directly from the Rabinowitz alternative (cf. Theorems \ref{th:Rabinowitz-alt-K} and \ref{th:unbounded-K}).
\begin{align*}
\sum_{(\alpha_0,0) \in \Lambda} \omega_G(\alpha_0) & = \omega_G(\alpha_{1,3,2}) + \omega_G(\alpha_{1,3,3}) + \omega_G(\alpha_{1,3,0}) + \omega_G(\alpha_{2,1,2}) + \omega_G(\alpha_{2,1,3}) 
\\ & = - 4(D_2^{\bz_1} \times^{\bz_1}D_1^p) - 2(D_2^{\bz_1} \times^{\bz_1}D_2^p)
- (D_2^{D_1} \times^{\bz_1}\bz_1^p)
+ 2(D_2^{\bz_1} 
\times^{\bz_2^-}D_4^z)
\\ & \quad 
+ 2(D_2^{\bz_1} \times^{\bz_2^-}V_4^p)
+ (D_2^{D_1} \times^{D_1^z}D_1^p)
+ (D_2^{D_1} \times^{\bz_2^-}D_2^p)
+ 2(D_6^{\bz_1} \times^{\bz_1}D_3^p)
\\ & \quad 
+ 2(D_2^{\bz_1} \times^{\bz_3}D_3^p)
+ (D_2^{D_1} \times^{\bz_3}\bz_3^p)
- 2(D_4^{\bz_1} \times^{\bz_2^-}D_4^p)
+ 2(D_2^{\bz_1} \times^{D_2^z}D_4^p)
\\ & \quad 
+ 2(D_2^{\bz_1} \times^{V_4}D_4^p)
- (D_2^{D_1} \times^{D_2^d}D_4^z)
+ (D_2^{D_1} \times^{D_2^z}D_4^z)
- (D_2^{D_1} \times^{D_3^z}D_3^p)
\\ & \quad 
- (D_2^{D_1} \times^{D_4^z}D_4^p)
+ (D_2^{D_1} \times^{D_4}D_4^p)
- (D_2^{D_1} \times^{D_4^{\hat{d}}}D_4^p)
- 2(D_6^{\bz_1} \times^{V_4}S_4^p)
\\ & \quad 
+ 4(D_6^{\bz_3} \times^{\bz_1}D_1^p)
+ 2(D_6^{\bz_3} \times^{\bz_1}D_2^p)
+ (D_6^{D_3} \times^{\bz_1}\bz_1^p)- 2(D_6^{\bz_3} \times^{\bz_2^-}D_4^z)
\\ & \quad 
- 2(D_6^{\bz_3} \times^{\bz_2^-}V_4^p)
- (D_6^{D_3} \times^{D_1^z}D_1^p)
- (D_6^{D_3} \times^{\bz_2^-}D_2^p)
- 2(D_18^{\bz_3} \times^{\bz_1}D_3^p)
\\ & \quad 
- 2(D_6^{\bz_3} \times^{\bz_3}D_3^p)
- (D_6^{D_3} \times^{\bz_3}\bz_3^p)
+ 2(D_12^{\bz_3} \times^{\bz_2^-}D_4^p)
-2(D_6^{\bz_3} \times^{D_2^z}D_4^p)
\\ & \quad 
- 2(D_6^{\bz_3} \times^{V_4}D_4^p)
+ (D_6^{D_3} \times^{D_2^d}D_4^z)
- (D_6^{D_3} \times^{D_2^z}D_4^z)
+ (D_6^{D_3} \times^{D_3^z}D_3^p)
\\ & \quad 
+ (D_6^{D_3} \times^{D_4^z}D_4^p)
- (D_6^{D_3} \times^{D_4}D_4^p)
+ (D_6^{D_3} \times^{D_4^{\hat{d}}}D_4^p)
+ 2(D_{18}^{\bz_3} \times^{V_4}S_4^p)
\\ & \quad 
+ (D_6^{D_3} \times^{S_4}S_4^p).
\end{align*}
In closing, let's consider a closer interpretation of the above results. Take the maximal orbit type $(H)= (D_{2}^{D_1}\times ^{D_4}D_4^p) \in \mathfrak M_{1,2}$ and let $\mathcal C$ be a branch of non-radial solutions with symmetries at least $({}^{3}H) = (D_{6}^{D_3}\times ^{D_4}D_4^p) \in \mathfrak M_{3,2}$ emerging from the first critical point $(\alpha_{1,3,2},0) \in \Lambda$. Notice that the kernel of $\mathscr A(\alpha): \mathscr H \rightarrow \mathscr H$ at any critical point $(\alpha_{n,m,j},0) \in \Lambda$ is given by the corresponding irreducible $G$-representation $\mathscr E_{n,m}^j$ (cf. \eqref{eq:enmj}), i.e. 
\[
\ker\mathscr A(\alpha_{1,3,2}) = \mathscr E_{1,3}^2
\]
such that any vector $\hat{u} \in \mathscr H$ with $\mathscr A(\alpha_{1,3,2}) \hat{u} = 0$ must be of the form
\[
\hat{u}(r,\theta) \in \left\{J_3(\sqrt{s_{13}}r)\Big(\cos(3\theta)\vec a+\sin(3\theta)\vec b\Big): \vec a,\, \vec b\in \text{span}\{\vec w_1, \vec w_2\} \right\}, 
\]
where $\vec w_1, \vec w_2 \in \mathcal V_2^-$ are the eigenvectors 
\[
\vec w_1 = (-1,1,-1,1,0,0)^T \text{ and } \vec w_2 = (1,0,1,0,-1,-1)^T.
\]
Computing the generators of the orbit type $(D_{6}^{D_3}\times ^{D_4}D_4^p)$,
\begin{align*}
\left(\frac{2\pi}{3}, (1), 1\right), \; \left(\kappa, (1), 1\right), \; \left(e_{SO(2)}, (1,2,3,4), 1\right), \; \left(e_{SO(2)}, (1,2), 1\right), \; \left(\frac{\pi}{3}, (1), -1\right) \in H
\end{align*}
and examining their action (cf. \eqref{def:isometric_action}, \eqref{def:S4-action}) on $\mathscr E_{1,3}^2$, we observe the radial invariances
\[
\hat u(r,\theta + \frac{2 \pi}{3}) = \hat u(r,\theta) \text{ and} - \hat u(r,\theta + \frac{\pi}{3}) = \hat u(r,\theta),
\]
which are satisfied trivially for all elements of $\mathscr E_{1,2}^3$, the conjugate invariance
\[
\hat u(r,-\theta) = \hat u(r,\theta),
\]
which is only satisfied when $\vec b = \vec 0$, and also the permutation invariances
\[
(1,2,3,4) \cdot \hat u(r,\theta) = \hat u(r,\theta) \text{ and } (1,2) \cdot \hat u(r,\theta) = \hat u(r,\theta),
\]
which, for $\vec a = a_1 \vec w_1 + a_2 \vec w_2$, is satisfied if and only if $a_2 = 2a_1$. Moreover, since one has
$\dim \ker\mathscr A(\alpha_0)^H = 1$, the celebrated theorem of Crandall-Rabinowitz (cf. \cite{CranR}) guarantees that the 
branch $\mathcal C$ is tangent to the solution $\hat{u}(r,\theta) = J_3(\sqrt{s_{13}}r)\cos(3\theta)(\vec w_1 + 2 \vec w_2)$ at the point $(\alpha_{1,3,2},0)$. Therefore, $\hat{u} \in \mathscr H$ can reasonably be used to estimate the behaviour of the branch $\mathcal C$ for parameter values near $\alpha_{1,3,2} \in \br$.
\begin{figure}[htbp]
    \centering
    \includegraphics[width=1\textwidth]{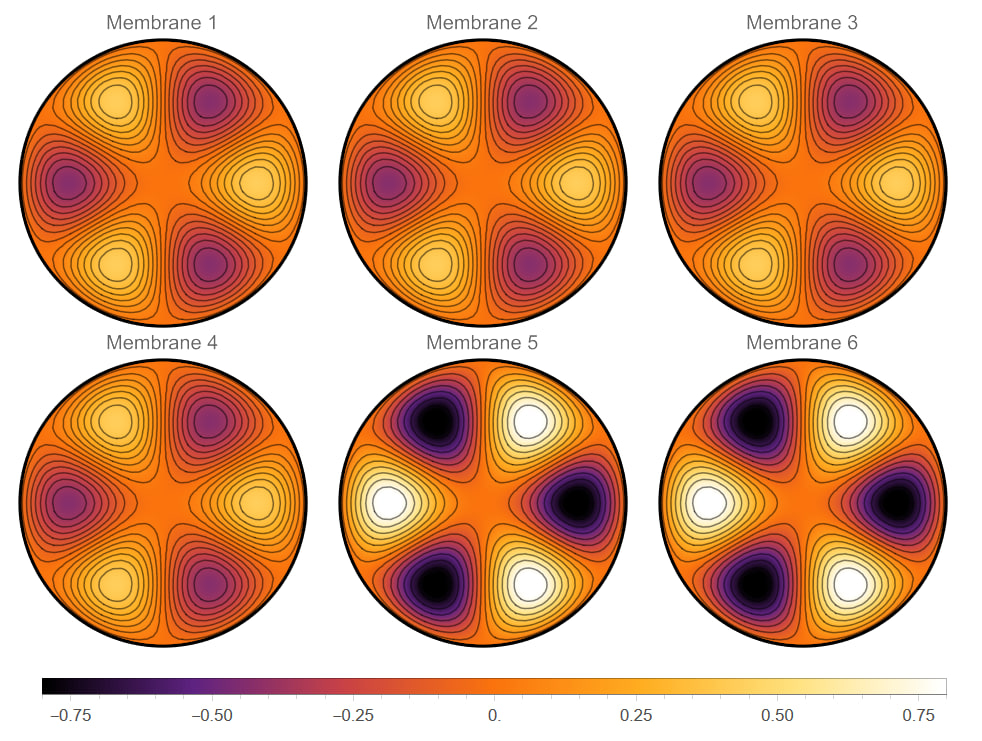}
    \caption{Graph of the map $\hat{u}(r,\theta) = J_3(\sqrt{s_{13}}r)\cos(3\theta)(\vec w_1 + 2 \vec w_2)$.}
    \label{fig:ker}
\end{figure}
\appendix
\section{Amalgamated Notation of Subgroups} \label{sec:amalgamated_notation}
The convention of \textit{amalgamated notation} is a shorthand for the specification of subgroups and their conjugacy classes in a product group $G_1 \times G_2$. In order to introduce this convention, we recall a well-known consequence of Goursat's Lemma (cf. \cite{Goursat}). Namely that, for any closed subgroup in $H \leq G_1 \times G_2$, there are two subgroups $K_1 \leq G_1$, $K_2 \leq G_2$, a group $L$ and a pair of epimorphisms $\varphi_2: K_1 \rightarrow L$, $\varphi_2: K_2 \rightarrow L$ such that
\begin{align} \label{not_amal_1}
H = \{ (x,y) \in K_1 \times K_2 \; : \; \varphi_1(x) = \varphi_2(y) \}.
\end{align}
Using amalgamated notation, the subgroup  \eqref{not_amal_1} is identified, up to its conjugacy class in $\Phi_0(G_1 \times G_2)$,  by the quintuple $(K_1, K_2, L, Z_1, Z_2)$, where $Z_1:= \ker \varphi_1$ and $Z_2:= \ker \varphi_2$, as follows
\begin{align*}
(H) = 
(K_1 {}^{Z_1}\times^{Z_2}_{L} K_2).
\end{align*}
In particular, one can always choose the group $L$ such that $L \simeq K_1/Z_1$, in which case the conjugacy class of \eqref{not_amal_1} is identified with the quadruple $(K_1,K_1,Z_1,Z_2)$ as follows
\begin{align}
\label{not_amal_3} 
(H) = 
(K_1 {}^{Z_1}\times^{Z_2} K_2).
\end{align}
This {\it compact amalgamated decomposition} \eqref{not_amal_3} is the form with which subgroup conjugacy classes are identified in this paper. 
\vs
In addition to the notation defined above, the following list provides some ancillary shorthand used in this paper for the identification of subgroups in $S_4 \times \bz_2$ 
\begin{align*}
\bz_2^-&=\{((1),1),((12)(34),-1)\},\\
\bz_4^-&=\{((1),1),((1324),-1),((12)(34),1),((1423),-1)\},\\
D_1^z&=\{((1),1),((12),-1)\},\\
V_4^-&=\{((1),1),((12)(34),1),((13)(24),-1),((14)(23),-1)\},\\
D_2^d&=\{((1),1),((12)(34),-1),((12),1),((34),-1)\},\\
D_2^z&=\{((1),1),((12)(34),1),((12),-1),((34),-1)\},\\
D_3^z&=\{((1),1),((123),1),((132),1),((12),-1),((23),-1),((13),-1)\}, \\
D_4^d&=\{((1),1),((1324),-1),((12)(34),1),((1423),-1),((34),1),\\
&~~~~~((14)(23),-1),((12),1),((13)(24),-1)\},\\
D_4^{\hat d}&=\{((1),1),((1324),-1),((12)(34),1),((1423),-1),((34),-1),\\
&~~~~~((14)(23),1),((12),-1),((13)(24),1)\},\\
D_4^z&=\{((1),1),((1324),1),((12)(34),1),((1423),1),((34),-1),\\
&~~~~~((14)(23),-1),((12),-1),((13)(24),-1)\},\\
S_4^-&=\{((1),1),((12),-1),((12)(34),1),((123),1),((1234),-1),((13),-1),\\
&~~~~~((13)(24),1),((132),1),((1342),-1),((14),-1),((14)(23),1),((142),1),\\
&~~~~~((1324),-1),((23),-1),((124),1),((1243),-1),((24),-1),((134),1),\\
&~~~~~((1423),-1),((34),-1),((143),1),((1432),-1),((243),1),((234),1)\}.
\end{align*}
\section{Spectral Properties of Operator $\mathscr A$}\label{sec:L}
In this section, we leverage the spectral properties of the Laplace operator $\mathscr L: \mathscr H \rightarrow \mathscr H$ to describe the spectrum of 
\begin{align*}
    \mathscr A(\alpha):= \operatorname{Id} - \mathscr L^{-1} A(\alpha): \mathscr H \rightarrow \mathscr H.
\end{align*}
To begin, consider the following eigenvalue problem defined on the planar unit disc $D:=\{z\in \bc: |z|<1\}$ with Dirichlet boundary conditions
\begin{align*}
\begin{cases}
-\Delta u=\lambda u, \quad u \in \mathscr H, \lambda \in \br\\
u|_{\partial D}= 0
\end{cases}
\end{align*} 
In polar coordinates, the Laplacian $\Delta$ takes the form
\begin{align*}
\Delta u=\frac {\partial ^2 u} {\partial r^2} +\frac {1} r  \frac{\partial  u}{ \partial
r}+\frac 1 {r^{2}}\frac {\partial ^2 u} {\partial \theta^2},
\end{align*} 
Assuming that $u \in \mathscr H$ can be separated  $u(r,\theta)=R(r)\cdot T(\theta)$, the eigenvalue problem $-\Delta u=\lambda u$ becomes
\begin{align*}
\left(-R''-\frac {1}r R'\right) T-\frac1 {r^2} T''R&=\lambda RT.
\end{align*}
Rearranging terms, we obtain
\begin{align}\label{eq:sep-Lap}
\frac{r^2R''+rR'}R+\lambda r^2=-\frac{T''}{T}.
\end{align}
Since the left and right-hand sides of equation \eqref{eq:sep-Lap} depend on different variables, there must exist a constant $c$ such that 
\begin{align}
r^2R''+rR'+(\lambda r^2-c)R=0,\label{eq:bessel}\\
-T''= cT.\label{eq:spherilcal-L}
\end{align}
To ensure that $u(r,\theta)$ is single-valued and continuous on $D$, we impose the periodicity condition
$T(\theta + 2\pi) = T(\theta)$. Then, the angular part \eqref{eq:spherilcal-L} becomes the ordinary differential equation
\begin{align*}
\begin{cases}
T''(\theta) = -c T(\theta)\\
T(\theta + 2\pi) = T(\theta).
\end{cases}
\end{align*} 
which has eigenvalues $c=m^2$ for $m=0,1,2,\ldots$ with corresponding eigenfunctions $T(\theta)=cos(m\theta)$ and $T(\theta)=\sin(m\theta)$.
\vs
Notice that, as a consequence of solving for $T(\theta)$, we have determined the values of constant $c$:
\begin{align*}c=m^2, m=0,1,2,\dots\end{align*}
Then, equation \eqref{eq:bessel} becomes 
\begin{equation}\label{eq:bessel1}
r^2 R''+rR'+(\lambda r^2-m^2)R=0.
\end{equation}
Introducing the change of variables $\rho=\sqrt \lambda r$ and $\wt R(\rho)=R(\sqrt \lambda r )$ and substituting into \eqref{eq:bessel1}, we obtain the classical Bessel equation
\[
\rho^2\wt R''(\rho)+\rho\wt R'(\rho)+(\rho^2-m^2)\wt R(\rho)=0.
\]
The bounded at zero solution is given by $\wt R(\rho)=J_m(\rho)$, where $J_m$ stands for the $m$-th  Bessel function of the first kind. It follows that solutions to \eqref{eq:bessel} are constant multiples of  $R(r)=J_m(\sqrt \lambda  r)$ (notice that $R(r)$ is finite at zero).
\vs
Consequently, $u(r,\theta)=R(r)\Theta(\theta)$ satisfies the Dirichlet  condition if $R(1)=0$, i.e. $J_m(\sqrt \lambda )=0$. In this way, we obtain that the spectrum of the Laplace operator is given by
\[
\sigma(\mathscr L)=\{s_{nm}: \text{$\sqrt{s_{nm}}$ is  the $n$-th positive zero of  $J_m, \; n  \in \bn \; m = 0,1,\ldots\}$}.
\]
Moreover, corresponding to each eigenvalue $s_{nm} \in \sigma(\mathscr L)$, there is the eigenfunction
\[
u_{nm}(r,\theta) = J_m(\sqrt{s_{nm}}r) \left(A\cos(m\theta)+B\sin(m\theta) \right),
\]
where $A,B \in \br$, and also the eigenspace 
\begin{align*}
\label{eq:Egs}
\mathscr E(s_{nm})=\text{span\,}\Big\{J_m(\sqrt{s_{nm}}r)\Big(A\cos(m\theta)+B\sin(m \theta)\Big): \vec a,
\, \vec b\in \mathbb{R}^k, \; 0\le n \le k \Big\}
\end{align*}
In this way, we have determined the spectrum of our operator $\mathscr A$ 
\[
\sigma(\mathscr A)=\left\{\xi_{n,m,j}:=1-\frac {\mu_j(\alpha)}{s_{nm}} :  j=1,2,\dots,k,\; n\in \bn,\; m=0,1,2,\dots\right\}.
\]
\section{Equivariant Brouwer Degree Background}
\label{subsec:G-degree}
We encourage readers interested in exploring the statements made in this section with greater detail to refer to the texts \cite{book-new}, \cite{AED}.
\vs
\noi{\bf  Equivariant notation.} Let $G$ be a compact Lie Group.
For any subgroup $H \leq G$, denote by $(H)$ its conjugacy class, by
by $N(H)$ its normalizer by $W(H):=N(H)/H$ its Weyl group in $G$. The set of all subgroup conjugacy classes in $G$, denoted $\Phi(G):=\{(H): H\le G\}$, has a natural partial order defined as follows
\[
(H)\leq (K) \iff \exists_{ g\in G}\;\;gHg^{-1}\leq K.
\]
In particular, we put $\Phi_0 (G):= \{ (H) \in \Phi(G) \; : \; \text{$W(H)$  is finite}\}$ and, for any $(H),(K) \in \Phi_0(G)$, we denote by $n(H,K)$ the number of subgroups $\tilde K \leq G$ with $\tilde K \in (K)$ and $H \leq \tilde K$.
\vs
For a $G$-space $X$ and $x\in X$, denote by
$G_{x} :=\{g\in G:gx=x\}$  the {\it isotropy group}  of $x$
and call $(G_{x}) \in \Phi(G)$  the {\it orbit type} of $x\in X$. Put $\Phi(G,X) := \{(H) \in \Phi_0(G) \; : \; 
(H) = (G_x) \; \text{for some $x \in X$}\}$ and  $\Phi_0(G,X):= \Phi(G,X) \cap \Phi_0(G)$. For a subgroup $H\leq G$, the subspace $
X^{H} :=\{x\in X:G_{x}\geq H\}$ is called the {\it $H$-fixed-point subspace} of $X$. If $Y$ is another $G$-space, then a continuous map $f : X \to Y$ is said to be {\it $G$-equivariant} if $f(gx) = gf(x)$ for each $x \in X$ and $g \in G$.  \cite{AED}.
\vs
\noi{\bf The Burnside Ring and Axioms of Equivariant Brouwer Degree.}
The free $\bz$-module $A(G) := \bz[\Phi_0(G)]$ becomes the Burnside ring when equipped with a natural multiplicative operation, called the Burnside ring product and defined for any pair of generators $(H),(K) \in \Phi_0(G)$ as follows
\begin{align} \label{def:burnside_product}
    (H) \cdot (K) := \sum\limits_{(L) \in \Phi_0(G)} n_L(L), 
\end{align}
where the coefficients $n_L \in \bz$ are given by the recurrence formula
\begin{align} \label{def:recurrence_formula_coefficients_burnside_product}
    n_L := \frac{n(L,H) |W(H)| n(L,K) |W(K)| - \sum_{(\tilde L) > (L)} n_{\tilde L} n(L,\tilde L) |W(\tilde L)|}{|W(L)|}.
\end{align}
Let $V$ be an orthogonal $G$-representation with an open bounded $G$-invariant set $\Om \subset V$. A $G$-equivariant map $f:V \rightarrow V$ is said to be $\Om$-admissible if $f(x) \neq 0$ for all $x \in \partial \Om$, in which case the pair $(f,\Om)$ is called an admissible $G$-pair in $V$. We denote by $\mathcal M^G(V)$ the set of all admissible $G$-pairs in $V$ and by $\mathcal{M}^{G}$ the set of all admissible $G$-pairs defined by taking a union over all orthogonal $G$-representations, i.e.
\[
\mathcal M^G := \bigcup\limits_V \mathcal M^G(V).
\]
The following statement is the standard axiomatic definition of the {\it $G$-equivariant Brouwer degree}:
\vs
\begin{theorem}
	\label{thm:GpropDeg} There exists a unique map $\gdeg:\mathcal{M}
	^{G}\to A(G)$, that assigns to every admissible $G$-pair $(f,\Omega)$ the Burnside Ring element
	\begin{equation}
		\label{eq:G-deg0}\gdeg(f,\Omega)=\sum_{(H) \in \Phi_0(G)}%
		{n_{H}(H)},
	\end{equation}
	satisfying the following properties:
	\begin{itemize}
		\item[] \textbf{(Existence)} If  $n_{H} \neq0$ for some $(H) \in \Phi_0(G)$ in \eqref{eq:G-deg0}, then there
		exists $x\in\Omega$ such that $f(x)=0$ and $(G_{x})\geq(H)$.

		\item[] \textbf{(Additivity)} 
  For any two  disjoint open $G$-invariant subsets
  $\Omega_{1}$ and $\Omega_{2}$ with
		$f^{-1}(0)\cap\Omega\subset\Omega_{1}\cup\Omega_{2}$, one has
		\begin{align*}
			\gdeg(f,\Omega)=\gdeg(f,\Omega_{1})+\gdeg
			(f,\Omega_{2}).
		\end{align*}

		\item[] \textbf{(Homotopy)} For any 
  $\Omega$-admissible $G$-homotopy, $h:[0,1]\times V\to V$, one has
		\begin{align*}
			\gdeg(h_{t},\Omega)=\mathrm{constant}.
		\end{align*}

		\item[] \textbf{(Normalization)}
  For any open bounded neighborhood of the origin in an orthogonal $G$-representation $V$ with the identity operator $\id:V \rightarrow V$, one has
		\begin{align*}
			\gdeg(\id,\Omega)=(G).
		\end{align*}
	\end{itemize}
 \vs
The following are additional properties of the map $\gdeg$ which can be derived from the four axiomatic properties defined above:		
\begin{itemize}
		\item[] {\textbf{(Multiplicativity)}} For any $(f_{1},\Omega
		_{1}),(f_{2},\Omega_{2})\in\mathcal{M} ^{G}$,
		\begin{align*}
			\gdeg(f_{1}\times f_{2},\Omega_{1}\times\Omega_{2})=
			\gdeg(f_{1},\Omega_{1})\cdot \gdeg(f_{2},\Omega_{2}),
		\end{align*}
		where the multiplication `$\cdot$' is taken in the Burnside ring $A(G )$.

		\item[] \textbf{(Recurrence Formula)} For an admissible $G$-pair
		$(f,\Omega)$, the $G$-degree \eqref{eq:G-deg0} can be computed using the
		following Recurrence Formula:
		\begin{equation}
			\label{eq:RF-0}n_{H}=\frac{\deg(f^{H},\Omega^{H})- \sum_{(K)>(H)}{n_{K}\,
					n(H,K)\, \left|  W(K)\right|  }}{\left|  W(H)\right|  },
		\end{equation}
		where $\left|  X\right|  $ stands for the number of elements in the set $X$
		and $\deg(f^{H},\Omega^{H})$ is the Brouwer degree of the map $f^{H}%
		:=f|_{V^{H}}$ on the set $\Omega^{H}\subset V^{H}$.
	\end{itemize}
\end{theorem}
\vs
\noi{\bf Computation of Brouwer equivariant degree.} For an orthogonal $G$-representation $V$ with the open unit ball $B(V):=\left\{  x\in V:\left|  x\right|  <1\right\}$, denote by $\{ \mathcal V_i \}_{i \in \bn}$ the set of its irreducible $G$-subrepresentations. In particular, define the \emph{basic degree associated with $\mathcal V_i$} as follows
\begin{align*}
	\deg_{\mathcal{V}_{i}}:=\gdeg(-\id,B(\mathcal{V} _{i})),
\end{align*} 
Now, consider a $G$-equivariant linear isomorphism $T:V\to V$ and assume that $V$
has a $G$-isotypic  decomposition
\[
V = \bigoplus_{i \in \bn} V_i,
\]
where each isotypic component $V_i$ is equivalent to $m_i \in \bn$ copies of the irreducible $G$-representation $\mathcal V_i$. From the
Multiplicativity property of the $G$-equivariant Brouwer Degree, one has
\begin{align*}
  \gdeg(T,B(V))=\prod_{i \in \bn} \gdeg
	(T_{i},B(V_{i}))= \prod_{i \in \bn}\prod_{\mu\in\sigma_{-}(T)} \left(
	\deg_{\mathcal{V} _{i}}\right)  ^{m_{i}(\mu)}%
\end{align*}
where $T_{i}=T|_{V_{i}}$ and $\sigma_{-}(T)$ denotes the real negative
spectrum of $T$. \vskip.3cm

Notice that the basic degrees can be effectively computed from \eqref{eq:RF-0}: 
\begin{align*}
	\deg_{\mathcal{V} _{i}}=\sum_{(H)}n_{H}(H),
\end{align*}
where 
\begin{align*}
    n_{H}=\frac{(-1)^{\dim\mathcal{V} _{i}^{H}}- \sum_{H<K}{n_{K}\, n(H,K)\, \left|  W(K)\right|  }}{\left|  W(H)\right|  }.
\end{align*}
\vs
The following fact is well-known (see for example \cite{survey}).
\begin{lemma}
    For any irreducible $G$-representation $\mathcal V$, the basic degree $\deg_{\mathcal V} \in A(G)$ is an involutive element of the Burnside Ring, i.e.
    \[
    (\deg_{\cV_j})^2=\deg_{\cV_j} \cdot \deg_{\cV_j}=(G).
    \]
\end{lemma}
\vs

\end{document}